\documentclass[final]{amsart}

\usepackage{pdfsync}
\usepackage{amsmath,amssymb}
\usepackage{enumerate}
\usepackage{xspace}
\usepackage{boxedminipage}
\usepackage{algorithm}
\usepackage{algpseudocode}
\usepackage{listings}
\usepackage{tabularx}
\usepackage{subcaption}
\usepackage{tristan}
\usepackage{fullpage}
\usepackage[backend=biber,style=numeric,giveninits, maxbibnames=9]{biblatex}
\newcounter{custom}
\numberwithin{custom}{section} 

\let\oldsubsection\subsection
\renewcommand{\subsection}{\stepcounter{custom}\oldsubsection}

\theoremstyle{plain}

\newtheorem{theorem}[custom]{Theorem}
\newtheorem{lemma}[custom]{Lemma}
\newtheorem{proposition}[custom]{Proposition}
\newtheorem{corollary}[custom]{Corollary}
\newtheorem{remark}[custom]{Remark}
\newtheorem{assumption}[custom]{Assumption}

\newtheorem{example}[custom]{Example}

\renewcommand{\vec}[1]{\geovec{#1}}
\renewcommand{\mat}[1]{\geomat{#1}}

\renewcommand{\avg}[1]{\{\kern-1.2mm\{#1\}\kern-1.2mm\}}
\newcommand{\tnorm}[1]{{\left\vert\kern-0.25ex\left\vert\kern-0.25ex\left\vert #1 
    \right\vert\kern-0.25ex\right\vert\kern-0.25ex\right\vert}}

\numberwithin{equation}{section}
\setboolean{showtodo}{true}
\setboolean{shownotes}{true}
\setboolean{showchanges}{false}
\setboolean{usemathrsfs}{false}
\usepackage{tikz}
\usetikzlibrary{calc}
\usepackage{pgfplots}
\pgfplotsset{width=7cm,compat=1.14}
\usepgfplotslibrary{colorbrewer}

\newcommand{\ConvergenceTriangle}[4]{%
    \coordinate (A) at (axis cs:#1, #2); 
    \coordinate (B) at (axis cs:{#1*#3}, #2); 
    \coordinate (C) at (axis cs:{#1*#3}, {#2/#3^#4)}); 

    \draw[thick] (A) -- (B);
    \draw[thick] (A) -- (C);
    \draw[thick] (B) -- (C) node[midway, right] {#4};  
    
}

\definecolor{ptwo}{rgb}{1.0, 0.49, 0.0}
\definecolor{pthree}{rgb}{0.0, 0.5, 0.0}
\definecolor{pfour}{rgb}{0.36, 0.54, 0.66}

\title[]{Optimal control of a kinetic equation}

\author{
  Aaron Pim
}
\author{
  Tristan Pryer
}
\author{
  Alex Trenam
}
\address{
  \thanks{
    Department of Mathematical Sciences,
    University of Bath, Claverton Down, Bath BA2 7AY, UK.
    {\tt{A.R.Pim@bath.ac.uk}, \tt{tmp38@bath.ac.uk}, \tt{amt83@bath.ac.uk}}
}}

\begin{document}
\maketitle
\begin{abstract}
  This work addresses an optimal control problem constrained by a
  degenerate kinetic equation of parabolic-hyperbolic type. Using a
  hypocoercivity framework we establish the well-posedness of the
  problem and demonstrate that the optimal solutions exhibit a
  hypocoercive decay property, ensuring stability and
  robustness. Building on this framework, we develop a finite element
  discretisation that preserves the stability properties of the
  continuous system. The effectiveness and accuracy of the proposed
  method are validated through a series of numerical experiments,
  showcasing its ability to handle challenging PDE-constrained optimal
  control problems.
\end{abstract}

\section{Introduction}

Optimal control problems subject to partial differential equation
(PDE) constraints are fundamental in many scientific and engineering
applications. These problems are found in diverse fields such as fluid
flow \cite{DolgovStoll:2017}, heat conduction
\cite{HarbrechtKunothSimonciniUrban:2023}, structural optimisation
\cite{KolvenbachLassUlbrich:2018}, and radiotherapy
\cite{CoxHattamKyprianouPryer:2023}. In radioprotection and external
beam radiotherapy, for example, optimal control techniques are used to
determine the optimal configuration of radiation beams, maximising the
dose to the target (tumour) while minimising exposure to healthy
tissues
\cite{woolley2015optimal,AshbyChronholmHajnalLukyanovMacKenziePimPryer:2024}. The
accurate and efficient modelling and solution of these problems are
critical to advancements in such fields.

The mathematical framework for optimal control problems involving
elliptic and parabolic PDEs has been well established over the
years. This foundation includes theoretical results for both
distributed and boundary control problems
\cite{BieglerGhattasHeinkenschlossBloemen-Waanders:2003,
  ReesDollarWathen:2010, De-los-Reyes:2015}. Typically these problems
are formulated through first-order optimality systems, which introduce
adjoint states or Lagrange multipliers to express the necessary
conditions for optimality. These conditions then serve as the basis
for various numerical methods, including semismooth Newton methods \cite{Casas:2023},
gradient-based techniques, and methods like the alternating direction
method of multipliers (ADMM) \cite{CiarletMiaraThomas:1989,
  Bertsekas:1999}.

The continuous formulation of optimal control problems are often discretised using
numerical methods such as finite element methods
\cite{AllendesFuicaOtarola:2020, BrennerGedickeSung:2018}, spectral
methods \cite{RodenMills-WilliamsPearsonGoddard:2022}, and wavelet
methods \cite{Kunoth:2005}. Among these, finite element methods allow
flexibility in handling complex geometries and boundary conditions,
while spectral and wavelet methods offer advantages in problems with
high smoothness or localised features. The choice of discretisation
method depends on the problem at hand, the nature of the domain, and
the computational resources available.
The analysis of parabolic optimal control problems often proceeds from the stability properties possessed by the constraining parabolic PDE operator. For example, see the energy-regularised space-time discretisations of \cite{Langer:2021, Langer:2024} and the discontinuous-Galerkin-in-time schemes of \cite{Chrysafinos:2014}.

However, the challenges are significantly amplified when the PDE in
question is hyperbolic or degenerate parabolic, such as the Boltzmann
transport, Fokker--Planck--Boltzmann or Kolmogorov equations
\cite{lesaint1974finite, ackroyd1995foundations}. These equations are
often sensitive to initial conditions and lack the stability
properties associated with standard elliptic or parabolic problems,
making classical energy estimates ineffective. For example, even the
1d transport equation
\begin{equation}
  u_t + u_x = f,
\end{equation}
leads to stability estimates that deteriorate as the terminal time $T
\to \infty$. In contrast, problems that include even a moderate
diffusion term, such as
\begin{equation}
  u_t + u_x - \epsilon u_{xx} = f,
\end{equation}
exhibit stable behaviour due to the additional dissipation. These
challenges are amplified when considering a problem written as an
optimal control constrained by such a problem. This necessitates a
more robust framework for analysis and numerical approximation.

For degenerate parabolic problems, such as those described by
Kolmogorov-type equations, classical energy methods fail to provide
accurate stability estimates because the equations exhibit dissipation
in only some of the variables. This motivated the development of the
concept of \emph{hypocoercivity}, first formalised by Villani
\cite{villani2006hypocoercivity}, which provides a framework for
deriving stability and regularity estimates for equations that are
degenerate in certain directions. The hypocoercive framework uses
mixed derivative terms to establish regularity and dissipation even in
variables where diffusion is not explicitly present. For example, in
the context of kinetic equations, H\'erau and Nier \cite{NeirHerau}
showed that for the Fokker-Planck equation, the smallest positive
eigenvalue can be used to derive decay rates for the solution, while
Eckmann and Hairer \cite{EckmannHairer} provided estimates for
H\"ormander-hypoelliptic Cauchy problems.

The numerical treatment of hypocoercive systems is challenging,
particularly when designing methods that respect the stability
properties of the underlying continuous problem. Existing
stabilisation techniques
\cite{Burman:2014,BarrenecheaJohnKnobloch:2024} often involve the
addition of artificial diffusion to hyperbolic or non-coercive
problems, which can introduce numerical artifacts if not done
carefully. In recent years, Porretta and Zuazua
\cite{PorrettaZuazua:2017} developed a finite difference method
compatible with the hypocoercive framework, while Georgoulis
\cite{georgoulis2020hypocoercivitycompatible} introduced finite
element methods designed to respect hypocoercivity properties at the
discrete level.

The key contribution of this paper is the establishment of a
hypocoercive framework for the optimal control problem. This framework
provides amenable stability properties at the continuum level,
ensuring that the optimal control problem retains desirable stability
features despite the inherent challenges posed by the underlying
degenerate PDE constraint. Specifically, we apply the hypocoercive
framework to the Kolmogorov equation and, by constructing non-standard
energy arguments, we obtain stability estimates in line with those
derived for parabolic problems.

Once the hypocoercive framework is established for the optimal control
problem, the natural next step is the discretisation of the
system. For this we extend the finite element techniques developed in
\cite{georgoulis2020hypocoercivitycompatible} with the optimal control
setup. This leads to a finite element method that preserves the
stability properties at the discrete level, providing a robust
numerical framework for solving the optimal control problem. The
method respects the hypocoercive structure of the underlying continuum
problem, ensuring that the numerical solutions maintain the physical
properties of the system.

To demonstrate the effectiveness of our approach, we present a series
of numerical experiments. These experiments confirm that the method is
both stable and efficient, making it a viable solution framework for
optimal control problems involving degenerate PDE constraints.

The rest of the paper is setup as follows: In \S \ref{sec:notation} we
introduce the notation and model Kolmogorov problem. In \S
\ref{sec:hypo} we formally show the Kolmogorov problem is
hypocoercive. In \S \ref{sec:OC} we introduce the optimal control
problem and show the optimal solution also satisfies a hypocoercivity
property. In \S \ref{sec:fem} we recall the finite element method from
\cite{georgoulis2020hypocoercivitycompatible} and extend it to the
optimal control problem. Finally, in \S \ref{sec:numerics} we show
some numerical examples.

\section{Notation and problem setup}
\label{sec:notation}

 Throughout this work we denote the standard Lebesgue spaces by
 $\leb{p}(\omega)$, $1\le p\le \infty$,
 $\omega\subset\mathbb{R}^2$. The $\leb2$ inner product over $\omega$
 is denoted $\ltwop{\cdot}{\cdot}_{\omega}$, where the subscript is
 omitted when $\omega$ is the whole of the considered domain $\Omega$
 (defined below), with corresponding norm
 $\|\cdot\|_{\leb2(\omega)}$. We introduce the Sobolev spaces
 \cite[c.f.]{Evans2010,Renardy2004}
\begin{equation}
  \sobh{m}(\omega)
  := 
  \left\{w\in\leb2(\omega)
  : \D^{\vec\alpha}w\in\leb2(\omega), \text{ for } \abs{\vec\alpha}\leq m\right\},
\end{equation}
which are equipped with norms and semi-norms
\begin{gather}
  \|w\|_{\sobh{m}(\omega)}^2 
  := 
  \sum_{\abs{\vec \alpha}\leq m}\|\D^{\vec \alpha} w\|_{\leb2(\omega)}^2 
  \text{ and }
  \abs{w}_{\sobh{m}(\omega)}^2 
  =
  \sum_{\abs{\vec \alpha} = m}\|\D^{\vec \alpha} w\|_{\leb2(\omega)}^2
\end{gather}
respectively, where $\vec\alpha = \{ \alpha_1,...,\alpha_d\}$ is a
multi-index, $\abs{\vec\alpha} = \sum_{i=1}^d\alpha_i$ and
derivatives $\D^{\vec\alpha}$ are understood in a weak sense.
A zero subscript (e.g. $\sobh1_0(\omega)$) indicates vanishing trace on $\partial\omega$.

We extend the notation to include Bochner spaces for functions defined
on a time interval $I = (0, T]$ and spatial domain $\omega \subset
  \mathbb{R}^2$. For a Banach space $X$, we define the space of, e.g.,
  $\cont{0}$-functions with values in $X$ as
  \begin{equation}
    \cont{0}(I; X)
    :=
    \left\{ w : I \to X \mid w \text{ is continuous and } w(t) \in X \text{ for all } t \in I \right\}.
  \end{equation}
  The corresponding norm is given by
  \begin{equation}
    \|w\|_{\cont{0}(I; X)} 
    := 
    \sup_{t \in I} \|w(t)\|_X.
  \end{equation}
  This naturally extends to other spaces.

The problem domain consists of the position-velocity space $\Omega :=
(-x_{\max}, x_{\max}) \times (-1, 1)$, where $x_{\max} > 0$, and the
time interval $(0, T]$. We partition the position-velocity domain
  boundary $\partial \Omega$ into three components: the inflow
  boundary $\Gamma^-$, the outflow boundary $\Gamma^+$, and the
  parabolic boundary $\Gamma_0$, such that
\begin{align}
    \Gamma^- &:= \{-x_{\max}\} \times [0, 1) \cup \{x_{\max}\} \times (-1, 0], \\
    \Gamma^+ &:= \{-x_{\max}\} \times (-1, 0) \cup \{x_{\max}\} \times (0, 1), \\
    \Gamma_0 &:= (-x_{\max}, x_{\max}) \times \{\pm 1\},
\end{align}
as illustrated in Figure \ref{fig:domain}. We consider Dirichlet
boundary conditions on $\Gamma_0$ and $\Gamma^-$, motivating the
definitions
\begin{equation}
  \Gamma^-_0 := \Gamma^- \cup \Gamma_0, \qquad
  \Gamma^+_0 := \Gamma^+ \cup \Gamma_0.
\end{equation}
\begin{figure}[h!]
  \begin{center}
    \begin{tikzpicture}
    \def\xmax{2}

    \draw[thick, ->] (-\xmax - 1, 0) -- (\xmax + 1, 0) node[right] {$x$}; 
    \draw[thick, ->] (0, -1.5) -- (0, 1.5) node[above] {$v$}; 

    \draw[thick] (-\xmax, -1) rectangle (\xmax, 1);


    \draw[blue, ultra thick] (-\xmax, 0) -- (-\xmax, 1);  
    \draw[blue, ultra thick] (\xmax, -1) -- (\xmax, 0);  

    \draw[red, dashed, ultra thick] (-\xmax, -1) -- (-\xmax, 0);  
    \draw[red, dashed, ultra thick] (\xmax, 0) -- (\xmax, 1);    

    \node[blue] at (-\xmax - 0.3, 0.5) {$\Gamma^-$};
    \node[blue] at (\xmax + 0.3, -0.5) {$\Gamma^-$};
    \node[red] at (-\xmax - 0.3, -0.5) {$\Gamma^+$};
    \node[red] at (\xmax + 0.3, 0.5) {$\Gamma^+$};

    \draw[orange, ultra thick] (-\xmax, 1) -- (\xmax, 1); 
    \draw[orange, ultra thick] (-\xmax, -1) -- (\xmax, -1); 
    \node[orange] at (0.3, 1.2) {$\Gamma
    _0$};
    \node[orange] at (-0.3, -1.2) {$\Gamma_0$};

    \fill[blue, opacity=0.2] (-\xmax, 0) rectangle (-\xmax - 0.1, 1); 
    \fill[blue, opacity=0.2] (\xmax, -1) rectangle (\xmax + 0.1, 0); 
    \fill[red, opacity=0.2] (-\xmax, -1) rectangle (-\xmax - 0.1, 0); 
    \fill[red, opacity=0.2] (\xmax, 0) rectangle (\xmax + 0.1, 1); 

    \fill[orange, opacity=0.2] (-\xmax, 1) rectangle (\xmax, 1.1); 
    \fill[orange, opacity=0.2] (-\xmax, -1) rectangle (\xmax, -1.1); 

\end{tikzpicture}    
    \caption{\label{fig:domain} An illustration of the domain $\Omega$ and
      the relevant boundary regions.}
  \end{center}
\end{figure}
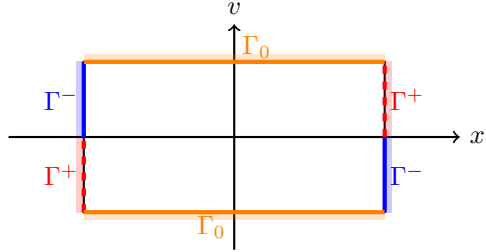

The $\leb2$ inner product with respect to the inflow and outflow
boundaries is defined as
\begin{equation}
  \ltwop{\phi}{\psi}_{\Gamma^{\pm}} := \int_{\Gamma^{\pm}}\phi \psi |v| \, \d s_{x} \d v,
\end{equation}
where $\d s_x$ denotes the surface measure in the $x$ direction, with
the associated norm $\|\cdot\|_{\Gamma^\pm}$.

Let $\epsilon > 0$ and $f$, $u_0$ be given. We then seek $u$ solving the Kolmogorov problem
\begin{align}
    u_t + v u_x - \epsilon u_{vv} &= f & \text{for }& t \in (0, T], (x, v) \in \Omega, \nonumber \\
    u(0, x, v) &= u_0 & \text{for }& (x, v) \in \Omega,  \label{eq:Kolmogorov} \\
    u(t, x, v) &= 0 & \text{for }& t \in (0, T], (x, v) \in \Gamma^-_0. \nonumber
\end{align}

\begin{remark}[Connection to Boltzmann transport]
  The Kolmogorov equation we study here is a prototypical example of
  the class of Boltzmann-Fokker-Planck kinetic equations, which arise
  in the description of charged particle transport. To illustrate
  this, suppose temporarily $X \subset \mathbb{R}^3$ is a closed,
  bounded spatial domain, let $\mathbb{S}^2$ be the unit sphere, and
  let $E = (\mathtt{e}_{\min}, \mathtt{e}_{\max}) \subset (0, \infty)$
  denote the energy domain. The space-angle-energy phase space is
  defined as
  \begin{equation}
    \Upsilon = X \times \mathbb{S}^2 \times E.
  \end{equation}
  For $(\vec{x}, \vec{\omega}, e) \in \Upsilon$, we define the
  outgoing part of the boundary as
  \begin{equation}
    \partial X^- := \{x \in \partial X : \vec{\omega} \cdot \vec{n}_X \leq 0\},
  \end{equation}
  where $\vec{n}_X$ is the outward-pointing normal to $\partial X$.

  The transport equation describing the fluence of charged particles,
  $\psi$, can be written as
  \begin{align}
    \partial_t \psi
    + \overbrace{\vec{\omega} \cdot \nabla_X \psi}  ^{\text{transport}}
    - \overbrace{\nabla_e (\varsigma \psi)} ^{\substack{\text{energy} \\ \text{depletion}}}
      - \overbrace{\epsilon \Delta_{\vec{\omega}} \psi}^{\substack{\text{elastic} \\ \text{scattering}}}
      &= \overbrace{f}^{\substack{\text{source / non-local} \\ \text{scattering}}}
      &&\quad \text{in } (0, T] \times \Upsilon, \nonumber \\
\psi(0, \vec{x}, \vec{\omega}, e) &= \psi_0(\vec{x}, \vec{\omega}, e) &&\quad \text{in } \{0\} \times \Upsilon, \\
\psi &= 0 &&\quad \text{on } (0, T] \times \partial X^-. \nonumber
  \end{align}
  In this form, the structural similarities with the Kolmogorov
  equation can be observed, particularly in the degenerate parabolic
  nature of the problem.
\end{remark}

\section{Hypocoercivity for the Kolmogorov equation}
\label{sec:hypo}

The concept of hypocoercivity is important for understanding the
behaviour of degenerate elliptic or parabolic operators. These
operators, which may initially appear non-dissipative due to the
absence of explicit damping terms in some variables, can exhibit
dissipative behaviour through the construction of non-standard
energies \cite{villani2006hypocoercivity}. 

\subsection{The stationary Kolmogorov operator}
\label{sec:stationary-cont}
We begin by introducing the operator 
\begin{equation}
  \mathcal{K}u := vu_x - \epsilon u_{vv},
\end{equation}
which corresponds to the stationary part of the solution operator in
(\ref{eq:Kolmogorov}). The operator $\mathcal{K}$ is an example of a
degenerate elliptic operator that is hypocoercive. Let
\begin{equation}
  \vec{M} = 
  \begin{bmatrix}
    M_{11} & M_{12} \\
    M_{12} & M_{22}
  \end{bmatrix}
\end{equation}
be a symmetric, non-negative definite matrix with non-negative entries, and let us define, respectively, the minimum and maximum eigenvalue operators $\lambda_-, \lambda_+ : \mathbb{R}^{2\times2} \longrightarrow \mathbb{R}$. 
With $\mathcal{H}:\sobh1(\Omega) \rightarrow \leb2(\Omega, \mathbb{R}^3)$ given by
\begin{equation}
  \mathcal{H}
  :=
  \begin{pmatrix}
    \vec{I} \\ \sqrt{\vec{M}}\nabla
  \end{pmatrix},
\end{equation}
where $\vec{I}$ is the identity matrix, for $w, \phi \in \sobh1(\Omega)$, we introduce the bilinear form
\begin{equation}\label{eq:H-inner-product-definition}
  \ltwop{w}{\phi}_{\mathcal{H}}
  :=
  \ltwop{\mathcal{H}w}{\mathcal{H}\phi}
  =
  \ltwop{w}{\phi} + \ltwop{\nabla w}{\vec{M} \nabla \phi},
\end{equation}
with associated norm or semi-norm (depending on whether $\vec{M}$ is strictly positive definite or indefinite)
\begin{equation}\label{eq:mathcal-H-norm-def}
  \Norm{w}_{\mathcal{H}}^2 := \Norm{w}_{\leb2(\Omega)}^2 + \Norm{\sqrt{\vec{M}} \nabla w}_{\leb2(\Omega, \reals^2)}^2.
\end{equation}

\begin{lemma}[Norm equivalence]
\label{lem:norm-equiv}
    For $w\in\sobh1(\Omega)$, we have
    \begin{equation}
        C^-_{\rm{eq}}\|w\|_{\sobh1(\Omega)}^2
        \leq
        \|w\|_{\mathcal{H}}^2
        \leq
        C^+_{\rm{eq}}\|w\|_{\sobh1(\Omega)}^2,
    \end{equation}
    where $C^-_{\rm{eq}} := \min\qp{1, \lambda_-\qp{\vec{M}}}$ and $C^+_{\rm{eq}} := \max\qp{1, \lambda_+\qp{\vec{M}}}$. Thus if $\vec{M}$ is strictly positive definite, then the $\|\cdot\|_{\mathcal{H}}$ and $\|\cdot\|_{\sobh1(\Omega)}$ norms are equivalent.
\end{lemma}
\begin{proof}
    We may bound the second term in \eqref{eq:mathcal-H-norm-def} as
    \begin{equation}
    \label{eq:bound-grad-term}
        \lambda_-\qp{\vec{M}}\|\nabla w\|_{\leb2(\Omega)}^2
        \leq
        \|\sqrt{\vec{M}}\nabla w\|_{\leb2(\Omega)}^2
        \leq
        \lambda_+\qp{\vec{M}}\|\nabla w\|_{\leb2(\Omega)}^2.
    \end{equation}
    It then follows that
    \begin{equation}
        \|w\|_{\mathcal{H}}^2
        \leq 
        \|w\|_{\leb2(\Omega)}^2 + \lambda_+\qp{\vec{M}}\|\nabla w\|_{\leb2(\Omega)}^2
        \leq
        \max\qp{1, \lambda_+\qp{\vec{M}}}
        \|w\|_{\sobh1(\Omega)}^2,
    \end{equation}
    and similarly
    \begin{equation}
        \|w\|_{\mathcal{H}}^2
        \geq 
        \|w\|_{\leb2(\Omega)}^2 + \lambda_-\qp{\vec{M}}\|\nabla w\|_{\leb2(\Omega)}^2
        \geq
        \min\qp{1, \lambda_-\qp{\vec{M}}}
        \|w\|_{\sobh1(\Omega)}^2.
    \end{equation}
\end{proof}
We associate to $\mathcal{K}$ the bilinear form $a: \sobh2(\Omega)
\times \sobh2(\Omega) \rightarrow \mathbb{R}$, given by
\begin{equation}
  \label{eq:weakform}
  \begin{split}
    a(w, \phi) :=& \ltwop{vw_x}{\phi}_{\mathcal{H}} + \epsilon \ltwop{w_v}{\phi_v}_{\mathcal{H}} \\
    =& \ltwop{v\mathcal{H}w_x}{\mathcal{H}\phi} + \ltwop{w_x}{M_{12} \phi_x + M_{22} \phi_v} + \epsilon \ltwop{w_v}{\phi_v} + \epsilon \ltwop{\nabla w_v}{\vec{M} \nabla \phi_v},
  \end{split}
\end{equation}
where the latter form will be useful to us later. We may then pose the
problem: given $f\in\sobh1(\Omega)$, seek $u\in\sobh2(\Omega)$, such
that
\begin{equation}
\label{eq:stationary-weak}
a(u, \phi) = \ltwop{f}{\phi}_{\mathcal{H}}\quad\forall\phi\in\sobh2(\Omega),
\end{equation}
subject to appropriate boundary conditions, which we specify in Remark
\ref{rem:boundary-conditions}.  Observe that if $\vec{M} \equiv \vec{0}$,
then \eqref{eq:stationary-weak} is simply a restriction to
$\sobh2(\Omega)$ of the standard weak formulation of the stationary
version of \eqref{eq:Kolmogorov}.

\begin{remark}[Alternative interpretation]
  An alternative interpretation of \eqref{eq:stationary-weak} starts
  from the differential operator itself. A differential consequence of
  \begin{equation}
    \mathcal{K}u - f = 0,
  \end{equation}
  is, for smooth enough $f$ and $u$, given by
  \begin{equation}
    \div\left( \vec{M} \nabla \left( \mathcal{K}u - f \right) \right) = 0.
  \end{equation}
  This leads to the fourth-order PDE
  \begin{equation}
    \label{eq:differential-consequence}
    \mathcal{K}u - f
    -
    \div\left( \vec{M} \nabla \left( \mathcal{K}u - f \right) \right) = 0,
  \end{equation}
  which is the strong form of \eqref{eq:stationary-weak}, with appropriate boundary conditions.
\end{remark}

\begin{remark}[Boundary conditions]
\label{rem:boundary-conditions}
  Since the strong form \eqref{eq:differential-consequence} of the hypocoercive formulation \eqref{eq:stationary-weak} is a fourth-order PDE, additional boundary conditions are required to ensure well-posedness. In this case, we assume
  \begin{equation}
    \mathcal{H}u = \vec{0} \quad \text{on } \Gamma^-_0,
  \end{equation}
  which is equivalent to
  \begin{equation}
    \begin{split}
      u &= u_x = u_v = 0 \quad \text{on } \Gamma^-, \\
      u &= u_x = u_v = 0 \quad \text{on } \Gamma_0.
    \end{split}
  \end{equation}
  It is important to note that these are not the only possible boundary conditions. Periodic boundary conditions on $\Gamma_0$, or natural boundary conditions
  \begin{equation}
    u = M_{11} u_{xv} + M_{12} u_{vv} = 0 \quad \text{on } \Gamma_0,
  \end{equation}
  are also feasible and do not affect the analysis presented.
\end{remark}

In light of Remark \ref{rem:boundary-conditions}, we now show that the bilinear form $a(\cdot, \cdot)$ is coercive over the space
  \begin{equation}
    \sobhsubsc{2}{0,\Gamma_0^-}(\Omega)
    :=
    \left\{\left. w \in \sobh2(\Omega) \ \right| \ \mathcal{H}w = \vec 0 \text{ on } \Gamma_0^- \right\}.
  \end{equation}
  We introduce the matrix
  \begin{equation}
    \vec N :=
      \begin{bmatrix}
        M_{12} & M_{22}/2 \\ M_{22}/2 & \epsilon
      \end{bmatrix},
    \end{equation}
which we will use throughout the remainder of this work.

\begin{lemma}[Coercivity of the bilinear form]
\label{lem:primal-coerc-alt}
  Let $w\in\sobhsubsc{2}{0,\Gamma_0^-}(\Omega)$, and let $4 \epsilon M_{12} > M_{22}^2$, which guarantees $\lambda_-(\vec N) > 0$. Then
  \begin{equation}
    a(w,w)
    \geq
    \frac{1}{2}\Norm{\mathcal{H}w}_{\Gamma^+}^2
      +
      \lambda_-(\vec{N})
      \Norm{\nabla w}_{\leb2(\Omega)}^2
      +
      \epsilon \lambda_-(\vec M)
      \Norm{\nabla w_v}_{\leb2(\Omega)}^2.
  \end{equation}
\end{lemma}
\begin{proof}
  From the definition of $a(\cdot, \cdot)$ in (\ref{eq:weakform}), we have
  \begin{equation}
      a(w,w)
      =
      \ltwop{v\mathcal{H}w_x}{\mathcal{H}w}
      +
      M_{12}\Norm{w_x}_{\leb2(\Omega)}^2
      +
      M_{22}\ltwop{w_x}{w_v}
      +
      \epsilon \Norm{w_{v}}_{\leb2(\Omega)}^2
      +
      \epsilon \ltwop{\nabla w_v}{\vec M \nabla w_v},
  \end{equation}
  whereby integrating the first term by parts and applying the boundary conditions gives
  \begin{equation}
  \begin{split}
      a(w, w)
      &=
      \dfrac{1}{2}\Norm{\mathcal{H}w}^2_{\Gamma^+}
      +
      M_{12}\Norm{w_x}_{\leb2(\Omega)}^2
      +
      M_{22}\ltwop{w_x}{w_v}
      +
      \epsilon \Norm{w_{v}}_{\leb2(\Omega)}^2
      +
      \epsilon \ltwop{\nabla w_v}{\vec M \nabla w_v}
      \\
    &=
        \dfrac{1}{2}\Norm{\mathcal{H}w}^2_{\Gamma^+}
        +
        \ltwop{\nabla w}{ \vec N\nabla w}
        +
      \epsilon \ltwop{\nabla w_v}{\vec M \nabla w_v}.
      \end{split}
    \end{equation}
    The result is then a consequence of the positive
    (semi-)definiteness of $\vec{M}$ and the strict positive
    definiteness of $\vec{N}$ under the stated assumptions.
\end{proof}

\begin{lemma}[Continuous dependence on problem data]
\label{lem:stationary-stability}
Let $\vec{M}$ satisfy the conditions of Lemma
\ref{lem:primal-coerc-alt}, let $u\in\sobhsubsc{2}{\Gamma_0^-}(\W)$
solve problem \eqref{eq:stationary-weak}, and let $C_p$ be the
Poincar\'e constant associated with $\Omega$. Then
    \begin{equation}
        \|u\|_{\mathcal{H}} \leq \frac{1}{\delta}\|f\|_{\mathcal{H}},
    \end{equation}
    where $\delta := \qp{C_{\rm{eq}}^+}^{-1}\qp{1 + C_p^2}^{-1}\lambda_-\qp{\vec{N}}$.
\end{lemma}
\begin{proof}
    Combining the Poincar\'e inequality with Lemma \ref{lem:norm-equiv} and using Lemma \ref{lem:primal-coerc-alt} gives
    \begin{align}
        \|u\|_{\mathcal{H}}^2
        &\leq
        C_{\rm{eq}}^+\qp{1 + C_p^2}\|\nabla u\|_{\leb2(\Omega)}^2 \\
        &\leq
        C_{\rm{eq}}^+\qp{1 + C_p^2}\qp{\lambda_-\qp{\vec{N}}}^{-1}a(u, u).
    \end{align}
    Using \eqref{eq:stationary-weak} and an applying the Cauchy-Schwarz inequality then completes the proof.
\end{proof}

\begin{remark}[A specific regularisation]
  Throughout this work we consider a general family of admissible $\vec M$. We now give a specific example, which we utilise in the
  numerical experiments in \S\ref{sec:numerics}. For a given $m>0$, let
  \begin{equation}\label{equation: Specific M choice}
    \vec M := \epsilon \begin{bmatrix} m^3 & m^2 \\ m^2 & m
    \end{bmatrix}.
  \end{equation}
  Then
  \begin{equation}\label{spcific M matrix C+ equiv value}
  \begin{gathered}
    \lambda_+\qp{\vec{M}} = \epsilon m (m^2+1), \quad  \lambda_-\qp{\vec{M}} = 0,
    \quad C_{\rm{eq}}^+ = \max\qp{1, \epsilon m (m^2+1)}, \quad C_{\rm{eq}}^- = 0, \\
    \lambda_{\pm}\qp{\vec{N}} = \frac{\epsilon}{2}\qp{1 + m^2 \pm \sqrt{m^4 + m^2 + 1}}.
    \end{gathered}
  \end{equation}
\end{remark}

\subsection{The time dependent Kolmogorov equation}
\label{sec:time-dep-Kolmogorov}

The hypocoercive weak formulation of the time-dependent Kolmogorov
problem \eqref{eq:Kolmogorov} is stated as follows. For a given time
interval $I := (0,T)$ and function $f \in \cont{}\qp{I;
  \sobh1(\Omega)}$, seek
\begin{equation}
  u \in \left\{\left. w \in \cont{1}\qp{I; \sobhsubsc{2}{0,\Gamma_0^-}(\Omega)}~\right| ~ \Norm{w(t)}_{\mathcal{H}} \in \cont{1}\qp{I} \right\},
\end{equation}
such that
\begin{equation}
  \begin{split}
    \ltwop{u_t}{\phi}_{\mathcal{H}} + a(u, \phi)
    &= \ltwop{f}{\phi}_{\mathcal{H}} \quad \Foreach \phi \in \sobhsubsc{2}{0, \Gamma_0^-}(\Omega) \text{ and a.e. } t \in I, \label{eq:Kolmogorov Hypo} \\
    u(0, x, v) &= u_0(x, v) \quad \text{for } (x, v) \in \Omega.
  \end{split}
\end{equation}
The following result demonstrates that the solution to this problem is
uniformly stable in time.

\begin{proposition}[Time stability]
\label{prop:primal-stability-alt}
Let the conditions of Lemma \ref{lem:primal-coerc-alt} be satisfied,
and let $\delta$ be defined as in Lemma
\ref{lem:stationary-stability}. Suppose $u$ solves
\eqref{eq:Kolmogorov Hypo}.  Then, for $\tau \in \bar{I}$,
  \begin{equation}
    \|u(\tau)\|^2_{\mathcal{H}} \leq \max_{\tau = 0,T} \left\{ \exp\qp{-\delta \tau} \|u_0\|^2_{\mathcal{H}} + \dfrac{1 - \exp\qp{-\delta \tau}}{\delta^2} \|f\|_{\cont{}(\bar{I}; \mathcal{H})}^2\right\}.
  \end{equation}  
\end{proposition}
\begin{proof}
  Let $\phi = u$ in \eqref{eq:Kolmogorov Hypo} and apply Lemma
  \ref{lem:primal-coerc-alt} to find
  \begin{align}
    \dfrac{1}{2} \dfrac{\d}{\d t} \|u\|^2_{\mathcal{H}}
    \leq
    \ltwop{f}{u}_{\mathcal{H}}
    -
    \lambda_-\qp{\vec{N}}\|\nabla u\|_{\leb2(\Omega)}^2,
  \end{align}
  whereby combining the Poincar\'e inequality with Lemma
  \ref{lem:norm-equiv} gives
  \begin{equation}
    \dfrac{1}{2} \dfrac{\d}{\d t} \|u\|^2_{\mathcal{H}}
    \leq
    \ltwop{f}{u}_{\mathcal{H}}
    -
    \lambda_-\qp{\vec{N}}\qp{C_p^2 + 1}^{-1}\qp{C_{\rm{eq}}^+}^{-1}\|u\|_{\mathcal{H}}^2.
  \end{equation}
  Recalling the definition of $\delta$ from Lemma \ref{lem:stationary-stability},
  we apply the Cauchy-Schwarz and Young inequalities to obtain
  \begin{equation}
  \label{eq:primal-delta-bound-ineq}
  \dfrac{1}{2} \dfrac{\d}{\d t} \|u\|^2_{\mathcal{H}}
  +
  \dfrac{\delta}{2} \|u\|^2_{\mathcal{H}}.
  \leq
  \dfrac{1}{2\delta} \|f\|^2_{\mathcal{H}}
  \end{equation}
  For arbitrary $\tau \in \bar{I}$, integrating and using Gr\"onwall's
  inequality yields
  \begin{equation}
    \|u(\tau)\|^2_{\mathcal{H}} \leq \exp\qp{-\delta \tau} \|u_0\|^2_{\mathcal{H}} + \dfrac{1}{\delta} \int_0^\tau \exp\qp{-\delta(\tau - s)} \|f(s)\|^2_{\mathcal{H}} \d s.
  \end{equation}
  To derive a uniform estimate, we bound $\|f(s)\|_{\mathcal{H}}$ from above by the Bochner norm $\|f\|_{\cont{}(\bar{I}; \mathcal{H})}$ to obtain
  \begin{equation}
    \|u(\tau)\|^2_{\mathcal{H}} \leq \exp\qp{-\delta \tau} \|u_0\|^2_{\mathcal{H}} + \dfrac{1 - \exp\qp{-\delta \tau}}{\delta^2} \|f\|_{\cont{}(\bar{I}; \mathcal{H})}^2.
  \end{equation}
  If $\delta^{-2} \|f\|_{\cont{}(\bar{I}; \mathcal{H})}^2 > \|u_0\|^2_{\mathcal{H}}$, then the right-hand side is monotonically decreasing with respect to $\tau$ and is maximised at $\tau = 0$. Similarly, if $\delta^{-2} \|f\|_{\cont{}(\bar{I}; \mathcal{H})}^2 < \|u_0\|^2_{\mathcal{H}}$, then the right-hand side is monotonically increasing and maximised at $\tau = T$.
\end{proof}

\begin{remark}[Comparing against a vanilla energy bound]
A non-hypocoercive weak formulation of \eqref{eq:Kolmogorov} reads: given $f\in\leb2(\Omega)$, seek $u \in \cont{1}(I;\sobh1(\Omega))$ such that
  \begin{equation}
  \label{eq:non-hypo-weak}
    \ltwop{u_t}{\phi}
    +
    \ltwop{vu_x}{\phi}
    +
    \epsilon \ltwop{u_v}{\phi_v} = \ltwop{f}{\phi}
    \Foreach \phi \in \sobh1(\Omega),
  \end{equation}
  subject to appropriate boundary conditions.
  Applying a standard energy argument to \eqref{eq:non-hypo-weak} 
  yields, for $\delta^* > 0$, the estimate
  \begin{equation}
    \dfrac{\d}{\d t} \|u\|_{\leb2(\Omega)}^2 + \|u\|^2_{\leb{2}(\Gamma^+)} + 2\epsilon \|u_v\|_{\leb2(\Omega)}^2 - \delta^* \|u\|_{\leb2(\Omega)}^2
    \leq
    \dfrac{1}{\delta^*} \|f\|_{\leb2(\Omega)}^2,
  \end{equation}
  and it then follows that
  \begin{equation}
    \dfrac{\d}{\d t} \|u\|_{\leb2(\Omega)}^2 - \delta^* \|u\|_{\leb2(\Omega)}^2 \leq \dfrac{1}{\delta^*} \|f\|_{\leb2(\Omega)}^2.
  \end{equation}
  Applying Gr\"onwall's inequality gives the pessimistic bound
  \begin{equation}
  \label{eq:pessimistic-bound}
    \|u(\tau)\|_{\leb2(\Omega)}^2 \leq \exp\qp{\delta^* \tau} \|u_0\|_{\leb2(\Omega)}^2 + \dfrac{1}{\delta^*} \int_0^\tau \exp\qp{\delta^*(\tau - t)} \|f(t)\|_{\leb2(\Omega)}^2 \, \d t,
  \end{equation}
  which is similar to the pure convection case.  In contrast to the
  hypocoercive stability demonstrated in Proposition
  \ref{prop:primal-stability-alt}, the non-hypocoercive bound in
  \eqref{eq:pessimistic-bound} offers no control as the terminal time
  $T \rightarrow \infty$. It is this stability we wish to exploit in
  the design of stable numerical methods for the constrained optimal
  control problem.
\end{remark}

\begin{remark}[An explicit Poincar\'e constant]
  The definition of $\delta$ in Lemma \ref{lem:stationary-stability} involves the Poincar\'e constant, which is generally unknown. However, it can be bounded by the diameter of the domain \cite{Optimal_Poincare_bound}. For a two-dimensional domain, we have
  \begin{equation}
    C_p \leq \dfrac{\text{diam}(\Omega)}{\pi} = \dfrac{2\sqrt{1 + x_{\max}^2}}{\pi} \qquad \implies \qquad \delta \geq \Tilde{\delta} :=  \qp{C_{\rm{eq}}^+}^{-1}\qp{4\pi^{-2}(1 + x_{\max}^2) + 1}^{-1}\lambda_-\qp{\vec{N}}.
  \end{equation}
  In practical implementation the value $\Tilde{\delta}$ can be used in place of $\delta$, since from \eqref{eq:primal-delta-bound-ineq}, we have
  \begin{equation}
    \dfrac{\d}{\d t}\|u\|^2_{\mathcal{H}} + \Tilde{\delta} \|u\|^2_{\mathcal{H}}
    \leq
    \dfrac{\d}{\d t} \|u\|^2_{\mathcal{H}} + \delta \|u\|^2_{\mathcal{H}}
    \leq \dfrac{1}{\delta} \|f\|^2_{\mathcal{H}} \leq \dfrac{1}{\Tilde{\delta}} \|f\|^2_{\mathcal{H}}.
  \end{equation}
\end{remark}

\section{Optimal control}
\label{sec:OC}

In this section we introduce an optimal control problem where the Kolmogorov equation serves as a constraint. By using the hypocoercive formulation \eqref{eq:Kolmogorov Hypo} developed in the previous section, we are able to show the problem is stable. Given a target function $\mathcal{D}$, we aim to find the optimal forcing term $f$ such that the solution $u$ to \eqref{eq:Kolmogorov Hypo} is ``close'' to $\mathcal{D}$.

\subsection{Stationary control}

We begin by considering the steady-state problem. For $\alpha > 0$ and $\mathcal{D} \in \sobh1(\Omega)$, define the functional $E:\sobhsubsc{2}{0,\Gamma_0^-}(\Omega) \times \sobh1(\Omega) \rightarrow [0, \infty)$ as
\begin{equation}\label{stationary hypocoercive cost functional}
  E(u, f)
  :=
  \dfrac{1}{2} \|u - \mathcal{D}\|^2_{\mathcal{H}}
  +
  \dfrac{\alpha}{2} \|f\|^2_{\mathcal{H}}.
\end{equation}
We seek a function pair $\qp{u^*, f^*} \in \sobhsubsc{2}{0,\Gamma_0^-}(\Omega)
\times \sobh1(\Omega)$ that minimises this functional, i.e.,
\begin{equation}\label{stationary forward 1}
  \left( u^*, f^* \right)
  =
  \argmin_{\qp{u, f} \in \sobhsubsc{2}{0,\Gamma_0^-}(\Omega) \times \sobh1(\Omega)}
  E(u, f),
\end{equation}
subject to the constraint that $(u, f)$ satisfies the stationary Kolmogorov equation
\begin{equation}\label{eq:stationary forward 2}
  a(u, \phi)
  =
  \ltwop{f}{\phi}_{\mathcal{H}} \qquad \forall
  \phi \in \sobhsubsc{2}{0, \Gamma_0^-}(\Omega).
\end{equation}

Defining the adjoint bilinear form
$a^*:\sobhsubsc{2}{0,\Gamma_0^+}(\Omega)\times
\sobhsubsc{2}{0,\Gamma_0^+}(\Omega) \rightarrow \mathbb{R}$ by
  \begin{equation}
  \label{eq:adjoint-def}
  \begin{split}
    a^*(z, \psi)
    &:=
    \ltwop{v\psi_x}{z}_{\mathcal{H}}
    +
    \epsilon\ltwop{\psi_v}{z_v}_{\mathcal{H}} \\
    &=
    -
    \ltwop{v \mathcal{H} z_x}{\mathcal{H} \psi}
    + \ltwop{M_{12} z_x + M_{22} z_v}{\psi_x}
    +
    \epsilon \ltwop{z_v}{\psi_v}
    +
    \epsilon \ltwop{\nabla z_v}{\vec{M}\nabla\psi_v},
    \end{split}
  \end{equation}
the following result details the first-order necessary conditions for a pair $\qp{u, f}$ to satisfy the optimal control problem \eqref{stationary forward 1}--\eqref{eq:stationary forward 2}.
  
\begin{lemma}[First-order necessary conditions]
  \label{Derivation of the optimal control}
  Suppose $\qp{u, z, f} \in \sobhsubsc{2}{0, \Gamma_0^-}(\Omega) \times \sobhsubsc{2}{0, \Gamma_0^+}(\Omega) \times \sobh1(\Omega)$ solves the following system:
  \begin{equation}
    \label{eq: Stationary KKT}
    \begin{split}
      a(u, \phi)
      &=
      \ltwop{f}{\phi}_{\mathcal{H}}, \\
      a^*(z, \psi)
      &=
      \ltwop{\mathcal{D} - u}{\psi}_{\mathcal{H}}, \\
      \ltwop{\alpha f - z}{\xi}_{\mathcal{H}} &= 0
      \Foreach \qp{\phi, \psi, \xi} \in \sobhsubsc{2}{0, \Gamma_0^-}(\Omega) \times \sobhsubsc{2}{0, \Gamma_0^+}(\Omega) \times \sobh1(\Omega).
    \end{split}
  \end{equation}
  Then $\qp{u, f}$ is the unique solution to the optimal control problem \eqref{stationary forward 1}--\eqref{eq:stationary forward 2}.
\end{lemma}

\begin{proof}
  The system of equations \eqref{eq: Stationary KKT} corresponds to the Karush-Kuhn-Tucker (KKT) conditions of the Lagrangian
  \begin{equation}\label{Lagrangian 2}
    \begin{split}
      L(u, f, z)
      &:=
      E(u, f)
      +
      \ltwop{vu_x}{z}_{\mathcal{H}} + \epsilon \ltwop{u_v}{z_v}_{\mathcal{H}} - \ltwop{f}{z}_{\mathcal{H}} \\
      &=
      E(u, f)
      - \ltwop{v \mathcal{H} z_x}{\mathcal{H} u}
      + \ltwop{M_{12} z_x + M_{22} z_v}{u_x}
      + \epsilon \ltwop{z_v}{u_v}_{\mathcal{H}} 
      - \ltwop{f}{z}_{\mathcal{H}}.
    \end{split}
  \end{equation}
  The first-order variations of $L$ yield the following conditions:
  \begin{align}
    \dfrac{\delta L}{\delta z} &= a(u, \phi) - \ltwop{f}{\phi}_{\mathcal{H}}, \label{eq: Primal Lagrangian} \\
    \dfrac{\delta L}{\delta u} &= \ltwop{u - \mathcal{D}}{\psi}_{\mathcal{H}} + a^*(z, \psi), \label{eq: Dual Lagrangian} \\
    \dfrac{\delta L}{\delta f} &= \ltwop{\alpha f - z}{\xi}_{\mathcal{H}}. \label{eq: Control Lagrangian}
  \end{align}
  The function spaces involved are closed, meaning that the extrema of the Lagrangian occur at the kernel of the first-order Lagrangian derivatives. Since the cost functional $E(u, f)$ is quadratic and the constraints are linear, the extrema correspond to the unique minima, concluding the proof.
\end{proof}

\begin{remark}[Optimisation then hypocoercification vs. hypocoercification then optimisation]
  In general, the adjoint of the hypocoercive primal equation does not
  coincide with the hypocoercive formulation of the adjoint for the
  original (non-hypocoercive) primal equation. This difference is
  captured (weakly) by the final term in equation
  \eqref{eq:adjoint-def}. Consequently, there is a choice in how to
  approach the optimal control problem, much like the distinction
  between discretising an optimal control problem after its derivation
  or deriving the problem from a discretisation
  (optimise-then-discretise vs. discretise-then-optimise). In this
  work, we choose to ``hypocoerce'' and then optimise (followed by
  discretising), though other approaches are valid.
\end{remark}

The stability of the optimal control problem \eqref{eq: Stationary KKT} relies on the coercivity of $a(\cdot, \cdot)$ over $\sobh2_{0, \Gamma_0^-}(\Omega)$ and $a^*(\cdot, \cdot)$ over the space
\begin{equation}
    \sobh2_{0, \Gamma_0^+}(\Omega)
    :=
    \left\{\left. w \in \sobh2(\Omega) \ \right| \ \mathcal{H}w = \vec 0 \text{ on } \Gamma_0^+ \right\}.
\end{equation}
The former was demonstrated in Lemma \ref{lem:primal-coerc-alt}. The latter is shown in the following.
\begin{lemma}[Coercivity of the adjoint bilinear form]
\label{lem:dual-coerc-alt}
  Let $z\in\sobh2_{0, \Gamma_0^+}(\Omega)$. Assume that the matrix $\vec{M}$ satisfies the same conditions as in Lemma \ref{lem:primal-coerc-alt}. Then
  \begin{equation}
    a^*(z, z) \geq \frac{1}{2}\|\mathcal{H}z\|_{\Gamma^-}^2 + \lambda_-\qp{\vec{N}}\|\nabla z\|_{\leb2(\Omega)}^2 + \epsilon\lambda_-\qp{\vec{M}}\|\nabla z_v\|_{\leb2(\Omega)}^2.
  \end{equation}
\end{lemma}
\begin{proof}
From the definition of $a^*(\cdot, \cdot)$ in \eqref{eq:adjoint-def}, we have
  \begin{equation}
    a^*(z, z) = 
    -\ltwop{v \mathcal{H}z_x}{\mathcal{H}z}
    + M_{12}\|z_x\|_{\leb2(\Omega)}^2
    + M_{22}\ltwop{z_x}{z_v}
    + \epsilon \|z_v\|_{\leb2(\Omega)}^2
    + \ltwop{\nabla z_v}{\vec{M}\nabla z_v}.
  \end{equation}
  Integrating the first term by parts and applying the boundary conditions, we obtain
  \begin{align}
    a^*(z, z) &= \dfrac{1}{2}\|\mathcal{H}z\|^2_{\Gamma^-}
    + M_{12}\|z_x\|_{\leb2(\Omega)}^2
    + M_{22}\ltwop{z_x}{z_v}
    + \epsilon \|z_v\|_{\leb2(\Omega)}^2
    + \ltwop{\nabla z_v}{\vec{M}\nabla z_v}.
  \end{align}
  The result then follows by the same analysis as in Lemma \ref{lem:primal-coerc-alt}.
\end{proof}

\begin{corollary}[Continuous dependence on problem data]
  \label{cor:stability-alt}
  Let the conditions of Lemma \ref{lem:primal-coerc-alt} be satisfied, and let $\delta$ be defined as in Lemma \ref{lem:stationary-stability}. If $\qp{u, z, f}\in\sobh2_{\Gamma_0^-}(\Omega)\times\sobh2_{\Gamma_0^+}(\Omega)\times\sobh2(\Omega)$ solve \eqref{eq: Stationary KKT}, then
  \begin{equation}
      \delta\qp{\|u\|_{\mathcal{H}}^2 + \frac{\alpha}{2}\|f\|_{\mathcal{H}}^2}
      \leq
      \frac{1}{2\alpha\delta}\|\mathcal{D}\|_{\mathcal{H}}^2.
  \end{equation}
\end{corollary}

\begin{proof}
  Setting $\phi = u$, $\psi = z$, and $\xi = u$ in \eqref{eq:
    Stationary KKT} and combining yields
  \begin{equation}
  \label{eq:coercivity-oc-stat}
    \alpha a(u, u) + a^*(z, z) = \ltwop{\mathcal{D}}{z}_{\mathcal{H}}.
  \end{equation}
  Using the Poincar\'e inequality and Lemma \ref{lem:norm-equiv} in combination with Lemma \ref{lem:dual-coerc-alt} and Lemma \ref{lem:primal-coerc-alt}, respectively, leads to
  \begin{align}
      \alpha\delta\|u\|_{\mathcal{H}}^2 &\leq \alpha a(u, u), \label{eq:primal-H-coerc} \\
      \delta\|z\|_{\mathcal{H}}^2 &\leq a^*(z, z). \label{eq:dual-H-coerc}
  \end{align}
  It then follows that
  \begin{equation}
      \delta\qp{\alpha\|u\|_{\mathcal{H}}^2 + \|z\|_{\mathcal{H}}^2}
      \leq
      \ltwop{\mathcal{D}}{z}_{\mathcal{H}}.
  \end{equation}
  Applying the Cauchy-Schwarz and Young-with-$\varepsilon$ inequalities, with $\varepsilon = \delta/2$, then gives
  \begin{equation}
  \label{eq:oc-stability-eq}
      \delta\qp{\alpha\|u\|_{\mathcal{H}}^2 + \frac{1}{2}\|z\|_{\mathcal{H}}^2}
      \leq
      \frac{1}{2\delta}\|\mathcal{D}\|_{\mathcal{H}}^2.
  \end{equation}
  Choosing $\xi = f$ in \eqref{eq: Stationary KKT}, it can be shown by an application of the Cauchy-Schwarz inequality that
  \begin{equation}
  \label{eq:f-z-bound}
      \alpha\|f\|_{\mathcal{H}} \leq \|z\|_{\mathcal{H}}.
  \end{equation}
  Substituting this into \eqref{eq:oc-stability-eq} and dividing by $\alpha$ then concludes the proof.
\end{proof}

\subsection{Space-time control}

We now consider the dynamic optimal control problem. Without loss of generality, we take $u_0 = 0$, and we then define the admissible function
spaces
\begin{equation}
  \begin{split}
    \mathcal{A}^- &= \left\{\left. u \in \cont{1}\qp{I; \sobhsubsc{2}{0,\Gamma_0^-}(\Omega)}~\right| ~ \Norm{u(t)}_{\mathcal{H}} \in \cont{1}\qp{I}, \ u|_{t = 0} = 0 \right\}, \\
    \mathcal{A}^+ &= \left\{\left. z \in \cont{1}\qp{I; \sobhsubsc{2}{0,\Gamma_0^+}(\Omega)}~\right| ~ \Norm{z(t)}_{\mathcal{H}} \in \cont{1}\qp{I}, \ z|_{t = T} = 0 \right\} .
  \end{split}
\end{equation}
For $\alpha > 0$ and $\mathcal{D} \in \cont{}\left(I; \sobh1(\Omega)\right)$, we define the cost functional $J^T: \mathcal{A}^- \times \cont{}\left(I; \sobh1(\Omega)\right) \rightarrow [0, \infty)$ by
\begin{equation}\label{hypocoercive cost functional}
  J^T(u, f) := \int_0^T E(u(t), f(t)) \, \d t.
\end{equation}
We then seek a function pair $\qp{u^*, f^*} \in \mathcal{A}^-\times \cont{}\left(I; \sobh1(\Omega)\right)$, such that
\begin{equation}\label{eq: minimisation problem}
  \left( u^*, f^* \right) = \argmin_{\qp{u, f} \in \mathcal{A}^- \times \cont{}\left(I; \sobh1(\Omega)\right)} J^T(u, f),
\end{equation}
subject to the constraint
\begin{equation}\label{eq: forward 2}
  \begin{split}
    \ltwop{u_t}{\phi}_{\mathcal{H}} + a(u, \phi) &= \ltwop{f}{\phi}_{\mathcal{H}}, \qquad \text{a.e. } t \in \bar{I}, \quad \forall \phi \in \mathcal{A}^-.
  \end{split}
\end{equation}
Necessary first-order optimality conditions for the satisfaction of
the optimal control problem \eqref{eq: minimisation
  problem}--\eqref{eq: forward 2} are provided by the following.

\begin{lemma}[KKT conditions for the dynamic optimal control problem]\label{lemma: relation on the dynamic system}
  The functions $(u, z, f) \in \mathcal{A}^- \times \mathcal{A}^+ \times \cont{}\left(I; \sobh1(\Omega)\right)$ that solve the optimal control problem \eqref{eq: minimisation problem}--\eqref{eq: forward 2} satisfy the following system of equations:
  \begin{equation}\label{eq: Dynamic KKT}
    \begin{split}
      \int_0^T \left( \ltwop{u_t}{\phi}_{\mathcal{H}} + a(u, \phi) \right) \d t - \int_0^T \ltwop{f}{\phi}_{\mathcal{H}} \d t &= 0, \quad \forall \phi \in \mathcal{A}^-, \\
      \int_0^T \left( - \ltwop{z_t}{\psi}_{\mathcal{H}}+a^*(z, \psi) \right) \d t + \int_0^T \ltwop{u - \mathcal{D}}{\psi}_{\mathcal{H}} \d t &= 0, \quad \forall \psi \in \mathcal{A}^+, \\
      \int_0^T \ltwop{\alpha f - z}{\xi}_{\mathcal{H}} \d t &= 0, \quad \forall \xi \in \cont{}\left(I; \sobh1(\Omega)\right).
    \end{split}
  \end{equation}
  Additionally, the functions $(u, z) \in \mathcal{A}^- \times
  \mathcal{A}^+$ satisfy the relation:
  \begin{equation}\label{eq: for+adj sum}
    \begin{split}
      \dfrac{\alpha}{2} \|u(T)\|^2_{\mathcal{H}} + \dfrac{1}{2} \|z(0)\|^2_{\mathcal{H}} + \alpha \int_0^T a(u, u) \d t + \int_0^T a^*(z, z) \d t
      &= \int_0^T \ltwop{\mathcal{D}}{z}_{\mathcal{H}} \d t.
    \end{split}
  \end{equation}
\end{lemma}

\begin{proof}
  By applying the same methodology as in Lemma \ref{Derivation of the
    optimal control}, we deduce that the minima of the dynamic problem
  satisfy the KKT conditions given in equation \eqref{eq: Dynamic
    KKT}. To derive the bound in equation \eqref{eq: for+adj sum} we
  notice that the control equation may be substituted into the primal
  equation, showing
  \begin{equation}\label{eq: Primal-Dual combined}
    \begin{split}
      \alpha \int_0^T \left( \ltwop{u_t}{\phi}_{\mathcal{H}} + a(u, \phi) \right) \d t &= \int_0^T \ltwop{z}{\phi}_{\mathcal{H}} \d t, \\
      \int_0^T \left(  - \ltwop{z_t}{\psi}_{\mathcal{H}} + a^*(z, \psi) \right) \d t &= -\int_0^T \ltwop{u}{\psi}_{\mathcal{H}} \d t + \int_0^T \ltwop{\mathcal{D}}{\psi}_{\mathcal{H}} \d t,
    \end{split}
  \end{equation}
  for all $\phi \in \mathcal{A}^-$ and $\psi \in \mathcal{A}^+$. Let $\phi = u$ and $\psi = z$, and sum the resulting equations to give
  \begin{equation}\label{eq: proof of dynamic KKT relation 1}
    \begin{split}
      \int_0^T \left( \alpha \ltwop{u_t}{u}_{\mathcal{H}} - \ltwop{z_t}{z}_{\mathcal{H}} + \alpha a(u, u) + a^*(z, z) \right) \d t &= \int_0^T \ltwop{\mathcal{D}}{z}_{\mathcal{H}} \d t.
    \end{split}
  \end{equation}
  Note that the ``cross-term'' $\ltwop{u}{z}_{\mathcal{H}}$ has been
  eliminated. 
  
  Next, observe that
  \begin{equation}
    \begin{split}
      2 \int_0^T \left( \alpha \ltwop{u_t}{u}_{\mathcal{H}} - \ltwop{z_t}{z}_{\mathcal{H}} \right) \d t
      &= \int_0^T \left( \alpha \frac{\d}{\d t} \|u\|^2_{\mathcal{H}} - \frac{\d}{\d t} \|z\|^2_{\mathcal{H}} \right) \d t \\
      &= \alpha \|u(T)\|^2_{\mathcal{H}} + \|z(0)\|^2_{\mathcal{H}}.
    \end{split}
  \end{equation}
  Substituting this relation into \eqref{eq: proof of dynamic KKT relation 1} completes the proof.
\end{proof}

\begin{remark}[Box control constraints]
  This approach can be extended to optimal control problems with additional constraints. For instance, consider the box constraint $f \in [0, 1]$.
  By applying the same variational methods, we obtain equations \eqref{eq: Primal Lagrangian} and \eqref{eq: Dual Lagrangian}, along with the following variational inequality:
  \begin{equation}
    \int\limits_{0}^T\ltwop{\alpha f - z}{\xi - f}_{ \mathcal{H}} \ \d t\geq 0 \qquad \forall \xi \in \cont{}\left(I; \sobh1(\Omega,[0,1])\right).
  \end{equation}
  Therefore, for the canonical projection operator $\mathcal{P}_{[0,1]}: \sobh1(\Omega) \rightarrow \sobh1(\Omega, [0, 1])$, we have
  \begin{equation}
    \dfrac{1}{\alpha} \mathcal{P}_{[0,1]} z = f \qquad \text{a.e. on } I \times \Omega.
  \end{equation}
  Since the operator norm of the projection is bounded, we can substitute the projection into the primal equation and apply the Cauchy-Schwarz inequality to obtain
  \begin{equation}
    \begin{split}
      \dfrac{\alpha}{2} \|u(T)\|^2_{\mathcal{H}} + \dfrac{1}{2} \|z(0)\|^2_{\mathcal{H}} + \alpha \int_0^T a(u, u) \d t + \int_0^T a^*(z, z) \d t
      &\leq \int_0^T \ltwop{\mathcal{D}}{z}_{\mathcal{H}} \d t.
    \end{split}
  \end{equation}
\end{remark}

The final result of this section then demonstrates the stability in time of the solution to the time-dependent hypocoercive optimal control problem \eqref{eq: minimisation problem}--\eqref{eq: forward 2}.

\begin{theorem}[Temporal decay]
\label{thm:oc-temporal-stability}
  Let the conditions of Lemma \ref{lem:primal-coerc-alt} be satisfied, and let $\delta$ be defined as in Lemma \ref{lem:stationary-stability}. Suppose the functions $u \in \mathcal{A}^-$ and $f \in \cont{}\left(I; \sobh1(\Omega)\right)$ satisfy \eqref{eq: Primal-Dual combined}. Then
  \begin{equation}
  \|u(T)\|^2_{\mathcal{H}} + \alpha \|z(0)\|^2_{\mathcal{H}}
    \leq
    \frac{1}{\alpha\delta}\int_0^T \exp\qp{-\delta\qp{T-t}}\|\mathcal{D}(t)\|^2_{\mathcal{H}} \d t.
  \end{equation}
\end{theorem}

\begin{proof}
  From Lemma \ref{lemma: relation on the dynamic system} we have
  \begin{equation}
    \dfrac{\alpha}{2} \|u(T)\|^2_{\mathcal{H}}
    +
    \dfrac{1}{2} \|z(0)\|^2_{\mathcal{H}}
    +
    \int_0^T \qp{\alpha a(u, u) + a^*(z, z)} \d t
    =
    \int_0^T \ltwop{\mathcal{D}}{z}_{\mathcal{H}} \d t,
  \end{equation}
  and applying a similar analysis to \eqref{eq:primal-H-coerc}--\eqref{eq:f-z-bound} leads to
  \begin{equation}
    \frac{\alpha}{2}\|u(T)\|^2_{\mathcal{H}}
    +
    \frac{1}{2}\|z(0)\|^2_{\mathcal{H}}
    +
    \int_0^T \delta\qp{\alpha\|u\|^2_{\mathcal{H}}
      +
      \frac{1}{2}\|z\|^2_{\mathcal{H}}} \d t
      \leq \frac{1}{2\alpha\delta}\int_0^T \|\mathcal{D}\|_{\mathcal{H}}^2 \d t.
  \end{equation}
  We may then apply Gr\"onwall's inequality to obtain the result.
\end{proof}

\begin{remark}[Pointwise in time]
  The result from Theorem \ref{thm:oc-temporal-stability} can be
  generalised to obtain a pointwise bound on the $\mathcal{H}$-norms
  of $u$ and $z$. For an arbitrary $\tau \in (0, T]$, consider the
functional $J^\tau: \mathcal{A}^- \times \cont{}\left(\bar{I};
\sobh1(\Omega)\right) \rightarrow \mathbb{R}$ defined as:
    \begin{equation}
        J^\tau(u, f) := \int_0^\tau E(u(t), f(t)) \, \d t.
    \end{equation}
    Applying the same analysis as in Theorem
    \ref{thm:oc-temporal-stability}, with the same terminal condition
    on $z$, we obtain the bound
    \begin{equation}
      \begin{split}
        \alpha \|u(\tau)\|^2_{\mathcal{H}} + \|z(T - \tau)\|^2_{\mathcal{H}}
        \leq
        \frac{1}{\alpha\delta}\int_0^\tau \exp\qp{-\delta\qp{\tau-t}}\|\mathcal{D}(t)\|^2_{\mathcal{H}} \d t.
      \end{split}
    \end{equation}
    Since this result holds for any $\tau \in (0, T]$, we obtain a
pointwise bound on the $\mathcal{H}$-norm of $u$ for each $\tau$ in
the time interval.
\end{remark}

\section{Hypocoercivity-preserving finite element methods}
\label{sec:fem}

In this section, we introduce the finite element method proposed in
\cite{georgoulis2020hypocoercivitycompatible}. We extend this method
to the optimal control formulation and demonstrate that the finite
element discretisation satisfies a discrete hypocoercivity result
consistent with Theorem \ref{thm:oc-temporal-stability}.

Let $\mathcal{T}$ be a regular subdivision of $\Omega$ into disjoint
simplicial or box-type (quadrilateral/hexahedral) elements $K$. We
assume that the subdivision $\mathcal{T}$ is shape-regular (see e.g.
p.124 in \cite{ciarlet}), and that $\bar{\Omega} = \cup_{K}\bar{K}$,
where the elemental faces are straight line segments (referred to as
\emph{facets}). Let $\mathcal{E}$ denote the union of all
1-dimensional facets associated with the subdivision $\mathcal{T}$,
including those on the boundary. We define the set of internal facets
as $\mathcal{E}_{int} := \mathcal{E} \setminus \partial \Omega$.

For an integer $r \geq 2$, let $\mathcal{P}_r(K)$ denote the set of
polynomials of total degree at most $r$ on element $K$, and let
$\mathcal{Q}_r(K)$ denote the set of tensor-product polynomials of
degree at most $r$ in each variable. We define the finite element
space as
\begin{equation}
  \fes^\pm_0 := \{ v \in \sobh 1_{0,\Gamma^\pm_0}(\Omega) : v|_{K} \in \mathcal{R}_{r}(K) \},
\end{equation}
where $\mathcal{R}_{r}(K) \in \{\mathcal{P}_{r}(K),
\mathcal{Q}_{r}(K)\}$. We emphasise that local bases in
$\mathcal{P}_r(K)$ can also be used for box-type elements, as
discussed in \cite{DGpoly1,DGpoly2,DGpolyparabolic}.  We also define
the broken Sobolev spaces $\sobh{m}(\Omega,\mathcal{T})$, with respect
to the subdivision $\mathcal{T}$, as
\begin{equation}
  \sobh{m}(\Omega,\mathcal{T})
  :=
  \{ u \in \leb2(\Omega) : u|_{K} \in \sobh{m}(K),\, K \in \mathcal{T} \}.
\end{equation}

Let $K^+$ and $K^-$ be two adjacent elements sharing a facet $e :=
\partial K^+ \cap \partial K^- \subset \mathcal{E}_{int}$, with
outward normal unit vectors $\vec{n}^+$ and $\vec{n}^-$ on $e$,
respectively. The associated velocity components of the normals are denoted $\vec{n}_v^+$ and $\vec{n}_v^-$. For a potentially discontinuous function $w:\Omega \to
\mathbb{R}$ across $\mathcal{E}_{int}$, we define $w^+ := w|_{e
  \subset \partial K^+}$ and $w^- := w|_{e \subset \partial K^-}$, and
define the jump operators
\begin{align}
  &\jump{w} := w^+ \vec{n}^+ + w^- \vec{n}^-, &&\jump{w}_v := w^+ \vec{n}_v^+ + w^- \vec{n}_v^-.
\end{align}
If $e \in \partial K \cap \partial \Omega$, we define $\jump{v} := v^+
\vec{n}$ and $\jump{v}_v := v^+
\vec{n}_v$. Additionally, we define $h_{K} := \text{diam}(K)$ and collect
the element diameters into the element-wise constant function ${\bf
  h}:\Omega \to \mathbb{R}$, where ${\bf h}|_{K} = h_{K}$, ${\bf h}|_e
= (h_{K^+} + h_{K^-})/2$ for $e \subset \mathcal{E}_{int}$, and ${\bf
  h}|_e = h_K$ for $e \subset \partial K \cap \partial \Omega$. We
assume that the families of meshes considered are locally
quasi-uniform.

We now proceed to build the discretisation based on the results in
\S\ref{sec:hypo}--\ref{sec:OC}.

\subsection{Stationary Kolmogorov}
To begin, we consider the stationary
primal problem given in \S\ref{sec:stationary-cont}.
For functions $W, \Phi \in \sobh 2(\Omega, \mathcal{T})$, we define the discrete bilinear form
\begin{equation}
\begin{split}
  \label{eq:discretebilinear}
  a_h(W, \Phi)
  &:=
  \sum_{K \in \mathcal{T}}
  \ltwop{vW_x}{\Phi}_K
  +
  \ltwop{\nabla(v W_x)}{\vec M \nabla \Phi}_K
  +
  \epsilon \ltwop{W_v}{\Phi_v}_K
  +
  \epsilon \ltwop{\nabla W_v}{\vec M \nabla \Phi_v}_K \\
  &= \sum_{K\in\mathcal{T}}
  \ltwop{v\mathcal{H}W_x}{\mathcal{H}\Phi}_K
  +
  \ltwop{W_x}{M_{12}\Phi_x + M_{22}\Phi_v}_K
  +
  \epsilon\ltwop{W_v}{\Phi_v}_K
  +
  \epsilon\ltwop{\nabla W_v}{\vec{M}\nabla \Phi_v}_K,
  \end{split}
\end{equation}
and consider the following finite element method: Given $f \in \sobh1(\Omega)$, find $U \in \fes^-_0$ such that
\begin{equation}
  \begin{split}
    \label{eq:spatial-disc}
    a_h(U, \Phi)
    +
    s_h(U, \Phi)
    &=
    \ltwop{f}{\Phi}_\mathcal{H}
    \Foreach \Phi \in \fes^-_0,
  \end{split}
\end{equation}
where $s_h(\cdot, \cdot)$ denotes an appropriate stabilisation term to
ensure consistency and stability of the method.

\begin{remark}[Consistent stabilisation]
    An example of a consistent stabilisation is given by
    \begin{equation}
      \label{eq:stab}
      \begin{split}
        s_h(U, \Phi)
        &:=
        s_{h,1}(U, \Phi)
        +
        s_{h,2}(U, \Phi) \\
        &=
        -\int_{\mathcal{E}_{int} \cup \Gamma^-}
        \Transpose{(v, 0)} \cdot
        \left(
          M_{11} \jump{U_x} \avg{\Phi_x}
          +
          M_{12} \jump{U_x} \avg{\Phi_v}
          +
          M_{12} \jump{U_v} \avg{\Phi_x}
          +
          M_{22} \jump{U_v} \avg{\Phi_v}
        \right) \d s \\
        &\quad
        -\int_{\mathcal{E}_{int} \cup \Gamma_0}
        \qp{\avg{\vec M \nabla U_v} \cdot \jump{\nabla \Phi}_v
        +
        \avg{\vec M \nabla \Phi_v} \cdot \jump{\nabla U}_v
        -
        \sigma \jump{\nabla U}_v \cdot \vec M \jump{\nabla \Phi}_v} \d s,
      \end{split}
    \end{equation}
    where $\sigma = C_\sigma r^2 \mathbf{h}^{-1}$ for some constant
    $C_\sigma$, independent of $\bf{h}$ and $r$. This stabilisation
    technique shares similarities with other methods. For instance,
    $s_{h,1}(\cdot, \cdot)$ has been used in the approximation of
    dispersive operators \cite{KarakashianMakridakis:2015}, and
    $s_{h,2}(\cdot, \cdot)$ resembles $C^0$ interior penalty methods
    for biharmonic operators \cite{BrennerSung:2005}.
\end{remark}

\begin{assumption}[The stabilisation choice]
  We do not impose any specific choice of stabilisation; however, in
  the numerical experiments, we use (\ref{eq:stab}). For the analysis,
  it is sufficient to assume three things:
  \begin{itemize}
  \item
    The stabilisation is
    consistent. That is, for all $w \in \sobhsubsc{2}{0,
      \Gamma_0^-}(\Omega)\cap \sobh{3}(\Omega, \mathcal T)$ we have
    \begin{equation}
      a_h(w, \Phi) + s_h(w, \Phi) = a(w,\Phi) \Foreach \Phi\in \fes_0^-.
    \end{equation}
  \item The stabilisation is equivalent to the nonconforming component
    of the norm, for any $W\in\fes_0^-$ that
    \begin{equation}
      c_0 \Norm{\sqrt{\sigma\vec M}\jump{\nabla W}_v}^2_{\leb{2}(\mathcal{E}_{int} \cup \Gamma^-)}
      \leq
      s_h(W,W)
      \leq
      C_0 \Norm{\sqrt{\sigma\vec M}\jump{\nabla W}_v}^2_{\leb{2}(\mathcal{E}_{int} \cup \Gamma^-)}.    
    \end{equation}
    \item The stabilisation is bounded, for any $W,\Phi\in\fes_0^-$    
    \begin{equation}
      s_h(W,\Phi)
        \leq
        C_{BS}
     \Norm{\sqrt{\sigma\vec M}\jump{\nabla W}_v}_{\leb{2}(\mathcal{E}_{int} \cup \Gamma^-)}
     \Norm{\sqrt{\sigma\vec M}\jump{\nabla \Phi}_v}_{\leb{2}(\mathcal{E}_{int} \cup \Gamma^-)}.
    \end{equation}
      \end{itemize}
\end{assumption}

For $W\in\sobh{3}(\Omega, \mathcal T)$, we
introduce the discrete hypocoercive norm
\begin{equation}
    \tnorm{W}^2 := \|\mathcal{H}W\|_{\Gamma^+}^2 + \sum_{K\in\mathcal{T}}\left(\|\nabla W\|_{\leb2(K)}^2 + \|\nabla W_v\|_{\leb2(K)}^2\right) + \sigma\|\sqrt{\vec{M}}\jump{\nabla W}_v\|_{\leb2(\mathcal{E}_{int})}^2,
\end{equation}
under which the following results demonstrate the discrete problem is coercive.
\begin{lemma}[Discrete norm relations]
  \label{lem:equivalence}
  Let $W\in\fes^-_0$, and let $C_p$ be the Poincar\'e constant
  associated to $\Omega$. Then
  \begin{equation}
    \Norm{W}_{\mathcal{H}}^2
    \leq
    C_{\textup{eq},h}^+ \tnorm{W}^2,
  \end{equation}
  where $C_{\textup{eq},h}^+ := C_p^2 + \lambda_+\qp{\vec{M}}$.
\end{lemma}
\begin{proof}
    Recalling \eqref{eq:bound-grad-term} and applying the Poincar\'e inequality, we have
    \begin{align}
        \|W\|_{\mathcal{H}}^2
        &\leq
        \|W\|_{\leb2(\W)}^2 + \lambda_+\qp{\vec{M}}\|\nabla W\|_{\leb2(\W)}^2 \\
        &\leq
        \qp{C_p^2 + \lambda_+\qp{\vec{M}}}\|\nabla W\|_{\leb2(\W)}^2.
    \end{align}
    The result then follows.
\end{proof}

\begin{lemma}[The discrete problem is coercive]
\label{lem:discrete-primal-coerc}
Let $W\in\fes_0^-$, and let $\vec{M}$ be strictly positive definite
and such that $4 \epsilon M_{12} > M_{22}^2$, guaranteeing
$\lambda_-\qp{\vec{M}}, \lambda_-(\vec N) > 0$. Then, for $C_{\sigma}$
chosen large enough,
    \begin{equation}
        a_h(W, W) + s_h(W, W)
        \geq
        C_c \tnorm{W}^2,
    \end{equation}
    where $C_c > 0$ is a coercivity constant.
\end{lemma}
\begin{proof}
From the definition of $a_h(\cdot, \cdot)$ in \eqref{eq:discretebilinear}, we have
  \begin{align}
      a_h(W, W)
      &= \sum_{K\in\mathcal{T}}
  \ltwop{v\mathcal{H}W_x}{\mathcal{H}W}_K
  +
  M_{12}\|W_x\|_{\leb2(K)}^2
  +
  M_{22}\ltwop{W_x}{W_v}_K \\
  &\qquad\qquad+
  \epsilon\|W_v\|_{\leb2(K)}^2
  +
  \epsilon\ltwop{\nabla W_v}{\vec{M}\nabla W_v}_K \nonumber \\
  &= \sum_{K\in\mathcal{T}}
  \ltwop{v\mathcal{H}W_x}{\mathcal{H}W}_K
  +
  \ltwop{\nabla W}{\vec{N}\nabla W}_K
  +
  \epsilon\ltwop{\nabla W_v}{\vec{M}\nabla W_v}_K.
  \end{align}
  Integrating the first term on the right-hand side by parts (elementwise) and using the boundary conditions leads to
  \begin{align}
      \sum_{K\in\mathcal{T}}
  \ltwop{v\mathcal{H}W_x}{\mathcal{H}W}_K &= \frac{1}{2}
      \|\mathcal{H}W\|_{\Gamma^+}^2
      + \frac{1}{2}
      \int_{\mathcal{E}_{int}\cup\Gamma^-}
      \Transpose{(v, 0)} \cdot \jump{M_{11} W_x^2 + 2 M_{12} W_x W_v + M_{22} W_v^2} \, \d s \\
      &= \frac{1}{2}
      \|\mathcal{H}W\|_{\Gamma^+}^2 - s_{h,1}(W, W).
  \end{align}
  Combining all terms then gives
  \begin{align}
      a_h(W, W) + s_h(W, W) &= \frac{1}{2}
      \|\mathcal{H}W\|_{\Gamma^+}^2
      +
  \ltwop{\nabla W}{\vec{N}\nabla W}_K
  +
  \epsilon\ltwop{\nabla W_v}{\vec{M}\nabla W_v}_K \\
  &\quad
      + \sigma\|\sqrt{\vec{M}}\jump{\nabla W}_v\|_{\leb2(\mathcal{E}_{int})}^2
      -
      2 \int_{\mathcal{E}_{int}} \avg{\vec M \nabla W_v} \cdot \jump{\nabla W}_v \, \d s,
  \end{align}
  and the result follows from the assumptions on $\vec{M}$, in the same way as Lemma \ref{lem:primal-coerc-alt}. The stabilisation term ensures that the penalty terms provide
enough control on the jumps across internal facets. Choosing
$\sigma$ large enough guarantees that the negative contributions
from the jumps are dominated, concluding the proof.
\end{proof}

\begin{lemma}[Continuous dependence on problem data]
\label{lem:discrete-stationary-stability}
  Let the conditions of Lemma \ref{lem:discrete-primal-coerc} be
  satisfied, let $C_{\textup{eq},h}^+$ be defined as in Lemma
  \ref{lem:equivalence}, and let $C_c$ be the coercivity constant from
  Lemma \ref{lem:discrete-primal-coerc}. Suppose $U$ solves the
  problem \eqref{eq:spatial-disc}.  Let $\delta_h := C_c
  \qp{C_{\textup{eq}, h}^+}^{-1}$, then,
  \begin{equation}
    \Norm{U}_{\mathcal{H}}
    \leq
    \frac 1 {\delta_h}\Norm{f}_{\mathcal{H}}.
  \end{equation}
\end{lemma}
 \begin{proof}
  By Lemma \ref{lem:equivalence} and Lemma \ref{lem:discrete-primal-coerc}
  we have that
  \begin{equation}
    \delta_h \Norm{U}_{\mathcal H}^2
    \leq
    C_c \tnorm{U}^2
    \leq
    a_h(U,U) + s_h(U,U)
    =
    \ltwop{f}{U}_{\mathcal H}.
  \end{equation}
  An application of the Cauchy-Schwarz inequality then concludes the proof.
\end{proof}
This allows us to show a C\'ea result for the problem. 
\begin{lemma}[C\'ea's Lemma for steady state Kolmogorov]
\label{lem:cea}
Let $u \in \sobhsubsc{2}{0, \Gamma_0^-}(\Omega)\cap\sobh{3}(\Omega, \mathcal T)$
solve (\ref{eq:stationary-weak}) and let $U \in \fes^-_0$ solve the
discrete problem (\ref{eq:spatial-disc}). Then,
\begin{equation}
  \tnorm{u - U} \leq \qp{1+\frac{C_{\textup{eq}, h}^+}{C_c}} \inf_{\Phi \in \fes^-_0} \tnorm{u - \Phi}.
\end{equation}
\end{lemma}

\begin{proof}
  Using the coercivity result from Lemma
\ref{lem:discrete-primal-coerc}, we have for any $\Phi \in \fes^-_0$
\begin{equation}
  \begin{split}
    C_c \tnorm{U - \Phi}^2
    &\leq
    a_h(U-\Phi, U-\Phi) + s_h(U-\Phi, U-\Phi)
    \\
    &\leq
    a_h(U-u, U-\Phi) + s_h(U-u , U-\Phi)
    +
    a_h(u-\Phi, U-\Phi) + s_h(u-\Phi , U-\Phi).
  \end{split}
\end{equation}
As the discrete formulation is consistent we have Galerkin
orthogonality which then implies
\begin{equation}
    C_c \tnorm{U - \Phi}^2
    \leq
    a_h(u-\Phi, U-\Phi) + s_h(u-\Phi , U-\Phi),
\end{equation}
for any $\Phi \in \fes^-_0$. Using the boundedness of $a_h(\cdot,
\cdot)$ and $s_h(\cdot, \cdot)$, there exists a constant
$C_{\textup{eq}, h}^+$ such that
\begin{equation}
    C_c \tnorm{U-\Phi}^2 \leq C_{\textup{eq}, h}^+ \tnorm{u - \Phi} \tnorm{U-\Phi}.
\end{equation}
Dividing through by $\tnorm{U-\Phi}$, minimising over $\Phi \in
\fes^-_0$ and the triangle inequality then yields the desired result.
\end{proof}

\begin{corollary}[Best approximation for steady state Kolmogorov]
\label{cor:best-approx}
Under the conditions of Lemma \ref{lem:cea}, let $u \in
\sobhsubsc{2}{0, \Gamma_0^-}(\Omega) \cap \sobh{s}(\Omega)$ and $U \in
\fes^-_0$ solve the continuous and discrete problems,
respectively. Then,
\begin{equation}
  \tnorm{u - U}
  \leq
  \delta_h C_{\textup{approx}} h^{\min(r-1, s-2)} \norm{u}_{\sobh{s}(\Omega)},
\end{equation}
where $r$ is the degree of the polynomial basis used in $\fes^-_0$ and
$C_{\textup{approx}}$ is an approximation constant depending on the
mesh regularity.
\end{corollary}

\subsection{Spatial semi-discretisation of the time-dependent Kolmogorov}

The above method extends to the time-dependent case in a
straightforward manner via the semi-discretisation: Seek $U : [0, T]
\longrightarrow \fes^-_0$ such that, for a.e. $t \in (0, T]$,
\begin{equation}
\label{eq:primal-semi-disc}
    \begin{split}
      \ltwop{U_t}{\Phi}_{\mathcal{H}}
      +
      a_h(U, \Phi)
      +
      s_h(U, \Phi)
      &=
      \ltwop{f}{\Phi}_{\mathcal{H}} \qquad \Foreach \Phi \in \fes^-_0, \\
      U(0) &= \Pi u_0,
    \end{split}
\end{equation}
where $\Pi : \leb2(\Omega)\longrightarrow\fes^-_0$ denotes the
$\leb2$-projection onto the finite element space.

The following stability result follows the main argument of
Proposition \ref{prop:primal-stability-alt}, with discrete objects replacing
their continuous counterparts.

\begin{corollary}[Stability of the semi-discrete time-dependent problem]
\label{cor:semi-disc-stability}
Let the conditions of Lemma \ref{lem:discrete-primal-coerc} be
satisfied, and with $\delta_h = C_c \qp{C_{\textup{eq},
    h}^+}^{-1}$. Suppose $U\in\fes^-_0$ solves
\eqref{eq:primal-semi-disc}. Then
\begin{equation}
    \Norm{U(T)}^2_{\mathcal{H}}
    \leq
    \exp\left(-\delta_h T\right)
    \Norm{\Pi u_0}^2_{\mathcal{H}}
    +
    \frac{1}{\delta_h}\int_0^T \exp\left(-\delta_h\qp{T-t}\right)
    \Norm{f(t)}_{\mathcal{H}}^2 \, \d t.
\end{equation}
\end{corollary}

\begin{theorem}[{\cite[Theorem 5.6]{georgoulis2020hypocoercivitycompatible}}]
\label{theorem:georgoulis-56}
Let $u$ be the solution of (\ref{eq:Kolmogorov Hypo}). Assume that
$u_0, u, u_t \in \sobh{2}_{0, \Gamma_0^-}(\Omega) \cap
\sobh{s}(\Omega)$ for $s \geq 3$. Under the assumptions of Lemma
\ref{lem:discrete-primal-coerc}, the error $e := u - U$ of the finite
element method (\ref{eq:primal-semi-disc}) satisfies the bound
\begin{equation}
    \Norm{e(\tau)}_{\mathcal{H}}^2
    +
    \int_0^\tau
    \tnorm{e(t)}^2 \, \mathrm{d}t
    \leq
    C h^{2\min(s-2, r-1)}
    \Norm{u}_{H^1(0, \tau; \sobh{s}(\Omega))}^2,
\end{equation}
where $C$ is a constant independent of $h$, $t$, and
$\tau$ but depending on the regularity of $u$ and $u_t$.
\end{theorem}
\begin{proof}
  For the sake of this proof we denote $u$ the solution of the
  temporal problem (\ref{eq:Kolmogorov Hypo}), $U = U(t)$ the solution
  of the semi discrete approximation (\ref{eq:primal-semi-disc}) and
  $R u$ the steady state approximation given by
  (\ref{eq:spatial-disc}). This is since the steady state
  approximation is exactly the Ritz projection. This allows us to
  decouple the error into two parts, analogous to parabolic and
  elliptic components in, for example \cite[\S 1]{Thomee:2007}. To
  that end, let us denote 
  \begin{equation}
    e(t) = (U - u)(t) = \theta(t) + \rho(t)
    :=
    (U - R u)(t) + (R u - u)(t).
  \end{equation}
  By definition we have, for any $\Phi \in \fes^-_0$,
  \begin{equation}
    \begin{split}
      \ltwop{\theta_t}{\Phi}_{\mathcal{H}}
      +
      a_h(\theta, \Phi)
      +
      s_h(\theta, \Phi)
      &=
      \ltwop{f}{\Phi}_{\mathcal{H}}
      -
      \ltwop{(R u)_t}{\Phi}_{\mathcal{H}}
      -
      a_h(R u, \Phi)
      -
      s_h(R u, \Phi)
      \\
      &=
      \ltwop{(u - R u)_t}{\Phi}_{\mathcal{H}}
      \\
      &=
      -\ltwop{\rho_t}{\Phi}_{\mathcal{H}}.
    \end{split}
  \end{equation}
  Now taking $\Phi = \theta$ and using energy arguments we have
  \begin{equation}
    \frac 12 \ddt \Norm{\theta}_{\mathcal{H}}^2
    +
    C_c\tnorm{\theta}^2
    \leq
    \Norm{\rho_t}_{\mathcal{H}}\Norm{\theta}_{\mathcal{H}}.    
  \end{equation}
  Hence Gr\"onwall's inequality yields
  \begin{equation}
    \Norm{\theta(t)}_{\mathcal{H}}
    \leq
    \exp\qp{-2 \delta_h t}\Norm{U(0) - R u(0)}_{\mathcal{H}}
    +
    2 \int_0^t \exp\qp{-2 \delta_h (t - \tau)} \Norm{\rho_t(\tau)}_{\mathcal{H}}
    \d\tau.
  \end{equation}
  Using Lemma \ref{lem:equivalence} and Corollary
  \ref{cor:best-approx} we have for $u_t\in{\sobh{s}(\Omega)}$
  \begin{equation}
    \Norm{\rho_t}_{\mathcal{H}}
    =
    \Norm{(Ru - u)_t}_{\mathcal{H}}
    \leq
    C h^{\min(r-1, s-2)} \Norm{u_t}_{\sobh{s}(\Omega)},
  \end{equation}
  and
  \begin{equation}
    \Norm{\rho(t)}_{\mathcal{H}}
    =
    \Norm{R u - u}_{\mathcal{H}} \leq C h^{\min(r-1, s-2)} \Norm{u}_{\sobh{s}(\Omega)}.
  \end{equation}
  The initial condition can be bounded as
  \begin{equation}
    \Norm{U(0) - R u(0)}_{\mathcal{H}} \leq C h^{\min(r-1, s-2)} \Norm{u_0}_{\sobh{s}(\Omega)}.
  \end{equation}  
  Hence combining these bounds, we obtain the total error
  \begin{equation}
    \begin{split}
      \Norm{e(t)}_{\mathcal{H}}
      &\leq
      \Norm{\theta(t)}_{\mathcal{H}}
      +
      \Norm{\rho(t)}_{\mathcal{H}}
      \\
      &\leq C h^{\min(r-1, s-2)} \left(
      \exp(-2 \delta_h t) \Norm{u_0}_{\sobh{s}(\Omega)}     
      +
      \Norm{u(t)}_{\sobh{s}(\Omega)}
      +
      \int_0^t \exp(-2 \delta_h (t - \tau)) \Norm{u_t(\tau)}_{\sobh{s}(\Omega)} \d\tau
      \right).
    \end{split}
  \end{equation}
  The $\tnorm{e}$ term is treated similarly, concluding the proof. 
\end{proof}

\subsection{Stationary Kolmogorov constrained optimal control}

In this section, we extend the results of
\cite{georgoulis2020hypocoercivitycompatible} to the optimal control
system described in \S\ref{sec:OC}. With the cost functional defined
as
\begin{equation}
  E(U, F)
  :=
  \frac{1}{2} \Norm{U - \mathcal{D}}_{\mathcal{H}}^2
  +
  \frac{\alpha}{2} \Norm{F}_{\mathcal{H}}^2,
\end{equation}
we seek \( \qp{U^*, F^*} \in \fes^-_0 \times \fes \) such that
\begin{equation}
  \label{eq:discrete-oc}
  \qp{U^*, F^*}
  = \argmin_{\qp{U, F} \in \fes^-_0 \times \fes} E(U, F)
\end{equation}
subject to the constraint
\begin{equation}
  \label{eq:discrete-constraint}
  a_h(U, \Phi) + s_h(U, \Phi)
  =
  \ltwop{F}{\Phi}_{\mathcal{H}} \quad \Foreach \Phi \in \fes^-_0.
\end{equation}
We define the adjoint bilinear form \( a_h^*(Z, \Psi) \) as
\begin{equation}
  \label{eq:discretebilineardual}
  \begin{split}
  a_h^*(Z, \Psi)
  :=
  a_h(\Psi, Z)
  &=
  \sum_{K \in \mathcal{T}}
  \ltwop{v\Psi_x}{Z}_K
  +
  \ltwop{\nabla(v \Psi_x)}{\vec M \nabla Z}_K
  +
  \epsilon \ltwop{Z_v}{\Psi_v}_K
  +
  \epsilon \ltwop{\nabla Z_v}{\vec M \nabla \Psi_v} \\
  &= \sum_{K\in\mathcal{T}}
  \ltwop{v\mathcal{H}\Psi_x}{\mathcal{H}Z}_K
  +
  \ltwop{\Psi_x}{M_{12}Z_x + M_{22}Z_v}_K
  +
  \epsilon\ltwop{\Psi_v}{Z_v}_K
  +
  \epsilon\ltwop{\nabla \Psi_v}{\vec{M}\nabla Z_v}_K.
  \end{split}
\end{equation}
As in the primal case, there is flexibility in the choice of the
adjoint stabilisation $s_h^*(\cdot, \cdot)$. 
\begin{remark}[Consistent adjoint stabilisation]
  An example of a consistent adjoint stabilisation is given by
  \begin{equation}
    \label{eq:dualstab}
    \begin{split}
       s_h^*(z, \psi)
       &:=
       -
       \int_{\mathcal{E}_{int} \cup \Gamma_+}
       (v, 0)^T \cdot
       \qp{M_{11} \jump{\psi_x} \avg{z_x}
         +
         M_{12} \jump{\psi_x} \avg{z_v}
         +
         M_{12} \jump{\psi_v} \avg{z_x}
         +
         M_{22} \jump{\psi_v} \avg{z_v}}
       \, \d s
       \\
       &\quad
       -\int_{\mathcal{E}_{int} \cup \Gamma_0}
       \avg{\vec M \nabla z_v} \cdot \jump{\nabla \psi}_v        
       +
       \avg{\vec M \nabla \psi_v} \cdot \jump{\nabla z}_v
       -
       \sigma
       \jump{\nabla z}_v \cdot \vec M \jump{\nabla \psi}_v \, \d s,
    \end{split}
  \end{equation}
  where $\sigma = C_\sigma r^2 {\mathbf{h}}^{-1}$ for some constant
  $C_\sigma >0$, independent of $h$ and $r$. It is important to note
  that in the dual problem, the roles of inflow and outflow are
  reversed both within the finite element space and the stabilisation.
\end{remark}

\begin{assumption}[The dual stabilisation choice]
  As in the primal case, we do not impose any specific choice of
  adjoint stabilisation; however, in the numerical experiments, we use
  (\ref{eq:dualstab}). For the analysis,
  it is sufficient to assume three things:
  \begin{itemize}
  \item
    The stabilisation is consistent. That is, for all $z \in
    \sobhsubsc{2}{0, \Gamma_0^+}(\Omega)\cap \sobh{3}(\Omega, \mathcal
    T)$ we have
    \begin{equation}
      a_h^*(z, \Phi) + s_h^*(z, \Phi) = a^*(z,\Phi) \Foreach \Phi\in \fes_0^-.
    \end{equation}
  \item The stabilisation is equivalent to the nonconforming component
    of the norm, for any $Z\in\fes_0^+$ that
    \begin{equation}
      c_0 \Norm{\sqrt{\sigma\vec M}\jump{\nabla Z}_v}^2_{\leb{2}(\mathcal{E}_{int} \cup \Gamma^-)}
      \leq
      s_h(Z,Z)
      \leq
      C_0 \Norm{\sqrt{\sigma\vec M}\jump{\nabla Z}_v}^2_{\leb{2}(\mathcal{E}_{int} \cup \Gamma_0)}.    
    \end{equation}
    \item The stabilisation is bounded, for any $Z,\Psi\in\fes_0^-$    
    \begin{equation}
      s_h^*(Z,\Psi)
        \leq
        C_{BS}
     \Norm{\sqrt{\sigma\vec M}\jump{\nabla W}_v}_{\leb{2}(\mathcal{E}_{int} \cup \Gamma_0)}
     \Norm{\sqrt{\sigma\vec M}\jump{\nabla W}_v}_{\leb{2}(\mathcal{E}_{int} \cup \Gamma_0)}.
    \end{equation}
      \end{itemize}
\end{assumption}

\begin{lemma}[First order necessary conditions]
  Suppose $\qp{U, Z, F} \in \fes^-_0\times\fes^+_0\times\fes$ solve
  \begin{equation}
    \label{eq:discreteKKT}
    \begin{split}
      a_h(U, \Phi) + s_h(U, \Phi) &= \ltwop{F}{\Phi}_{\mathcal H}, \\
      a_h^*(Z, \Psi) + s_h^*(Z, \Psi) &= \ltwop{\mathcal{D} - U}{\Psi}_{\mathcal H}, \\
      \ltwop{\alpha F - Z}{\Xi}_{\mathcal H} &= 0 \qquad \Foreach \qp{\Phi, \Psi, \Xi} \in \fes^-_0 \times \fes^+_0 \times \fes.
    \end{split}
  \end{equation}
  Then $(U,F)$ are the unique solutions that satisfy
  \eqref{eq:discrete-oc}--\eqref{eq:discrete-constraint}.
\end{lemma}

\begin{proof}
  As in Lemma \ref{Derivation of the optimal control}, we observe that
  system \eqref{eq:discreteKKT} corresponds to the KKT conditions of
  the Lagrangian
  \begin{equation}
    L_h(U, F, Z) := E(U, F) + a_h(U, Z) + s_h(U, Z) - \ltwop{F}{Z}_{\mathcal H}.
  \end{equation}
  We compute the first variations of the Lagrangian as
  \begin{align}
    \frac{\delta L_h}{\delta Z}
    &=
    a_h(U, \Phi) + s_h(U, \Phi) - \ltwop{F}{\Phi}_{\mathcal H}, \label{eq:discrete-primal-lagrangian} \\
    \frac{\delta L_h}{\delta U}
    &=
    \ltwop{U - \mathcal{D}}{\Psi}_{\mathcal H} + a_h^*(Z, \Psi) + s_h^*(Z, \Psi), \label{eq:discrete-dual-lagrangian} \\
    \frac{\delta L_h}{\delta F}
    &=
    \ltwop{\alpha F - Z}{\Xi}_{\mathcal H}. \label{eq:discrete-control-lagrangian}
  \end{align}
  The function spaces $\fes^-_0 \times \fes^+_0 \times \fes$ are
  finite-dimensional and closed, ensuring that the extrema of the
  Lagrangian occur at the kernel of its first-order derivatives. Since
  the cost functional is quadratic and the constraints are linear, the
  system of KKT conditions \eqref{eq:discreteKKT} corresponds to the
  unique minimisers, completing the proof.
\end{proof}

For $Z\in\sobh{3}(\Omega, \mathcal{T})$, the discrete dual norm is
defined by
\begin{equation}
  \tnorm{Z}_*^2 := \|\mathcal{H}Z\|_{\Gamma^-}^2 + \sum_{K\in\mathcal{T}}\left(\|\nabla Z\|_{\leb2(K)}^2 + \|\nabla Z_v\|_{\leb2(K)}^2\right) + \sigma\|\sqrt{\vec{M}}\jump{\nabla Z}_v\|_{\leb2(\mathcal{E}_{int})}^2,
\end{equation}
under which the discrete adjoint problem is coercive. 
\begin{lemma}[Discrete dual norm equivalence]
  \label{lem:dual-equivalence}
  Let $Z\in\fes^+$, let $C_{\textup{eq},h}^+$ be defined as in Lemma
  \ref{lem:equivalence}. Then
  \begin{equation}
    \Norm{Z}_{\mathcal{H}}^2
    \leq
    C_{\textup{eq},h}^+ \tnorm{Z}_*^2.
  \end{equation}
\end{lemma}
\begin{proof}
  The result proceeds similarly to Lemma \ref{lem:equivalence}.
\end{proof}
\begin{lemma}[The discrete dual problem is coercive]
  Let $Z\in\fes_0^+$, let the conditions of Lemma
  \ref{lem:discrete-primal-coerc} be satisfied, and let $C_c$ be
  defined as in Lemma \ref{lem:discrete-primal-coerc}. Then, for
  $C_{\sigma}$ chosen large enough, we have
  \begin{equation}
    a_h^*(Z, Z) + s_h^*(Z, Z)
    \geq
    C_c \tnorm{Z}_*^2.
  \end{equation}
\end{lemma}
\begin{proof}
  The proof mimics that of Lemma \ref{lem:discrete-primal-coerc}, with
  the only difference being the boundary conditions and corresponding
  boundary terms in the stabilisation.
\end{proof}

\begin{corollary}[Stability]
  Let the conditions of Lemma \ref{lem:discrete-primal-coerc} be
  satisfied, and let $\delta_h$ be defined as in Lemma
  \ref{lem:discrete-stationary-stability}. Suppose $\qp{U, Z,
    F}\in\fes^-_0\times\fes_0^+\times\fes$ satisfy the discrete
  optimal control problem \eqref{eq:discreteKKT}. Then
  \begin{equation}
    C_c\qp{\tnorm{U}^2 + \frac{\alpha}{2} \tnorm{F}_*^2}
    \leq
    \frac{1}{2\alpha \delta_h} \|\mathcal{D}\|_{\mathcal{H}}^2.
  \end{equation}
\end{corollary}
\begin{proof}
  The proof mimics that of Corollary \ref{cor:stability-alt}. In this case, however, we do not apply the norm equivalence results at the beginning.
\end{proof}

\begin{lemma}[C\'ea's Lemma for steady-state Kolmogorov control]
  \label{lem:cea-optimal-control}
  Let $(u, z, f) \in \sobhsubsc{2}{0, \Gamma_0^-}(\Omega) \times
  \sobhsubsc{2}{0, \Gamma_0^+}(\Omega) \times \sobh{2}(\Omega)$ solve
  the stationary optimal control problem (\ref{eq: Stationary KKT})
  and let $(U, Z, F) \in \fes^-_0 \times \fes^+_0 \times \fes$ solve
  the finite element approximation (\ref{eq:discreteKKT}). Then,
  \begin{equation}
    \tnorm{U - u}^2 + \tnorm{Z - z}_*^2
    \leq
    C \inf_{(\Phi, \Psi) \in \fes^-_0\times \fes^+_0}
    \left( \tnorm{\Phi - u}^2 + \tnorm{\Psi - z}_*^2 \right),
\end{equation}
  where $C$ is a constant independent of $h$, $u$, $z$ and $f$.
\end{lemma}

\begin{proof}
Using the coercivity of the primal and adjoint bilinear forms, along
with the consistency of the discrete stabilisation terms, we have
for any $\Phi \in \fes^-_0$ and $\Psi \in \fes^+_0$
\begin{multline}
    \begin{split}
      C_c \qp{\tnorm{U - \Phi}^2
        + \tnorm{Z - \Psi}_*^2 }    
      &\leq
      a_h(U-\Phi, U-\Phi) + s_h(U-\Phi, U-\Phi)
      +
      a_h^*(Z-\Psi, Z-\Psi) + s_h^*(Z-\Psi, Z-\Psi)    
      \\
      &\leq
      a_h(U-u, U-\Phi) + s_h(U-u, U-\Phi)
      +
      a_h(u - \Phi, U-\Phi) + s_h(u-\Phi, U-\Phi)
      \\
      &\qquad +
      a_h^*(Z-z, Z-\Psi) + s_h^*(Z-z, Z-\Psi)
      +
      a_h^*(z-\Psi, Z-\Psi) + s_h^*(z-\Psi, Z-\Psi).
    \end{split}
\end{multline}
Using the consistency of the primal and adjoint discrete bilinear forms
and Galerkin orthogonality
\begin{multline}
    \begin{split}
      C_c \qp{\tnorm{U - \Phi}^2
        + \tnorm{Z - \Psi}_*^2 }    
      &\leq
      \ltwop{Z-z}{U-\Phi}_{\mathcal H}
      -
      \ltwop{U-u}{Z-\Psi}_{\mathcal H}
      \\
      &\qquad +
      a_h(u - \Phi, U-\Phi) + s_h(u-\Phi, U-\Phi)
      \\
      &\qquad +
      a_h^*(z-\Psi, Z-\Psi) + s_h^*(z-\Psi, Z-\Psi).
    \end{split}
\end{multline}
Notice that
\begin{multline}
    \begin{split}
      \ltwop{Z-z}{U-\Phi}_{\mathcal H}
      -
      \ltwop{U-u}{Z-\Psi}_{\mathcal H}
      &=
      \ltwop{Z-\Psi}{U-\Phi}_{\mathcal H}
      +
      \ltwop{\Psi-z}{U-\Phi}_{\mathcal H}
      \\
      &\qquad -
      \ltwop{U-\Phi}{Z-\Psi}_{\mathcal H}
      -
      \ltwop{\Phi - u}{Z-\Psi}_{\mathcal H}    
      \\
      &=    
      \ltwop{\Psi-z}{U-\Phi}_{\mathcal H}
      -
      \ltwop{\Phi - u}{Z-\Psi}_{\mathcal H}
      \\
      &\leq
      \Norm{\Psi-z}_{\mathcal H}
      \Norm{U-\Phi}_{\mathcal H}
      +
      \Norm{\Phi - u}_{\mathcal H}
      \Norm{Z-\Psi}_{\mathcal H}.
    \end{split}
\end{multline}
Applying Lemmata \ref{lem:equivalence} and \ref{lem:dual-equivalence}
and the Cauchy-Schwarz and Young inequalities
\begin{equation}
    \begin{split}
      \ltwop{Z-z}{U-\Phi}_{\mathcal H}
      -
      \ltwop{U-u}{Z-\Psi}_{\mathcal H}
      &\leq
      \delta_h^2
      \tnorm{\Psi-z}_*^2
      +
      \frac{C_c}{2}\tnorm{U-\Phi}^2
      +
      \delta_h^2
      \tnorm{\Phi - u}^2
      +
      \frac{C_c}{2} \tnorm{Z-\Psi}_*^2.
    \end{split}
\end{equation}
Substituting this back into the inequality
\begin{multline}
    \begin{split}
      \frac{C_c}{2} \qp{\tnorm{U - \Phi}^2 + \tnorm{Z - \Psi}_*^2 }    
      &\leq
      \frac{2 C^+_{eq,h}}{\delta_h} \qp{\tnorm{\Psi-z}^2 + \tnorm{\Phi - u}^2} \\
      &\quad +
      a_h(u - \Phi, U-\Phi) + s_h(u-\Phi, U-\Phi)
      \\
      &\quad +
      a_h^*(z-\Psi, Z-\Psi) + s_h^*(z-\Psi, Z-\Psi).
    \end{split}
\end{multline}
Using the boundedness of the bilinear forms and Young's inequality
\begin{equation}
  \begin{split}
    a_h^*(Z-z, Z-\Psi) + s_h^*(Z-z, Z-\Psi)
    &\leq
    \frac{2 C^+_{eq,h}}{\delta_h} \tnorm{Z-z}_*^2 + \frac{C_c}{4} \tnorm{Z-\Psi}_*^2, \\
    a_h^*(z-\Psi, Z-\Psi) + s_h^*(z-\Psi, Z-\Psi)
    &\leq
        \frac{2 C^+_{eq,h}}{\delta_h} \tnorm{z-\Psi}_*^2 + \frac{C_c}{4} \tnorm{Z-\Psi}_*^2.
  \end{split}
\end{equation}
Combining terms concludes the proof.
\end{proof}

\begin{corollary}[Best approximation for steady state control]
  \label{cor:best-approx-optimal-control}
  Under the conditions of Lemma \ref{lem:cea-optimal-control}, suppose
  $u \in \sobhsubsc{2}{0, \Gamma_0^-}(\Omega) \cap \sobh{s}(\Omega)$,
  $z \in \sobhsubsc{2}{0, \Gamma_0^+}(\Omega) \cap \sobh{s}(\Omega)$
  with $s\geq 3$. Then, the error satisfies
  \begin{equation}
    \qp{\tnorm{U - u}^2 + \tnorm{Z - z}_*^2}^{1/2}
    \leq
    C h^{\min(r-1, s-2)}
    \left( \Norm{u}_{\sobh{s}(\Omega)} + \Norm{z}_{\sobh{s}(\Omega)} \right),
  \end{equation}
  where $r$ is the degree of the polynomial basis and $C$ is a
  constant depending on the mesh regularity but independent of $h, u$
  and $z$.
\end{corollary}

\subsection{Spatial semi-discretisation of time-dependent Kolmogorov constrained optimal control}

We consider the following spatially discrete finite element
method. For a.e. $t\in(0, T]$, seek $\qp{U(t), Z(t), F(t)} \in
  \fes^-_0 \times \fes^+_0\times\fes$, such that
\begin{equation}
  \label{eq:control-spatial-disc}
  \begin{split}      
    \int_0^T
    \ltwop{\partial_t U}{\Phi}_{\mathcal{H}}
    +
    a_h(U, \Phi)
    +
    s_h(U, \Phi)
    -
    \ltwop{F}{\Phi}_{\mathcal{H}} \d t
    &= 0 \Foreach \Phi \in \fes^-_0
    \\
    \int_0^T
    - \ltwop{\partial_t Z}{\Psi}_{\mathcal{H}}
    +
    a_h^*(Z, \Psi)
    +
    s_h^*(Z, \Psi)
    + \ltwop{U - \mathcal{D}}{\Psi}_{\mathcal{H}} \d t
    &= 0 \Foreach \Psi \in \fes^+_0
    \\
    \int_0^T
    \ltwop{\alpha F - Z}{\Xi} \d t
    &= 0, \Foreach \Xi\in\fes \\
    U(0) &= 0 \\
    Z(T) &= 0.
  \end{split}
\end{equation}

The solutions to \eqref{eq:control-spatial-disc} satisfy the following
stability result, the proof of which follows the same methodology as
that of Theorem \ref{thm:oc-temporal-stability}.

\begin{theorem}[Stability of the time-dependant optimal control problem]
\label{thm:control-spatial-stability}
For a.e. $t \in (0, T]$, let $\qp{U(t), Z(t), F(t)} \in \fes^-_0 \times \fes^+_0 \times \fes$ solve \eqref{eq:control-spatial-disc}, and let $\delta_h$ be defined as in Lemma \ref{lem:discrete-stationary-stability}. Then,
\begin{equation}
  \Norm{U(T)}_{\mathcal{H}}^2
  +
  \frac{1}{\alpha} \Norm{Z(0)}_{\mathcal{H}}^2
  \leq
  \dfrac{1}{\delta_h} \int_0^T \exp\qp{-2\delta_h (T - t)} \Norm{\mathcal{D}(t)}^2_{\mathcal{H}} \d t.
\end{equation}
\end{theorem}

\begin{lemma}[A priori bound for time-dependent optimal control]
\label{lem:cea-fully-discrete-control}
Let $(u, z, f) \in H^1(0, T; \sobh{s}(\Omega)\cap\sobhsubsc{2}{0,
  \Gamma_0^-}(\Omega)) \times H^1(0, T;
\sobh{s}(\Omega)\cap\sobhsubsc{2}{0, \Gamma_0^+}(\Omega)) \times
L^2(0, T; \sobh{s}(\Omega))$ solve (\ref{eq: Dynamic KKT}) and let
$(U, Z, F) \in \fes^-_0 \times \fes^+_0 \times \fes$ solve
\eqref{eq:control-spatial-disc}. Then,
\begin{multline}
  \alpha\Norm{U(T) - u(T)}_{\mathcal{H}}^2
  +
  \Norm{Z(0) - z(0)}_{\mathcal{H}}^2
  +
  \delta_h \qp{
    \int_0^T \alpha \tnorm{U(t) - Ru(t)}^2 
    +
    \tnorm{Z(t) - Rz(t)}^2 \, \mathrm{d}t
  }
  \\
  \leq C h^{2\min(r-1, s-2)}
  \qp{
    \Norm{u}_{H^1(0, T; \sobhsubsc{s}{0, \Gamma_0^-}(\Omega))}^2
    +
    \Norm{z}_{H^1(0, T; \sobhsubsc{s}{0, \Gamma_0^+}(\Omega))}^2
  }.
\end{multline}
\end{lemma}

\begin{proof}
  As in the primal temporal problem, we let $\qp{U, Z} = \qp{U(t),
    Z(t)}$ denote the solution of (\ref{eq:control-spatial-disc}) and
  $\qp{Ru, Rz}$ the solution of the time independent case
  (\ref{eq:discreteKKT}). Let us decompose the errors
  \begin{equation}
    e_U(t) = (U - u)(t) = \theta_U(t) + \rho_U(t) = (U - Ru)(t) + (Ru - u)(t) 
  \end{equation}
  and
  \begin{equation}
    e_Z(t) := (Z - z)(t) = \theta_Z(t) + \rho_Z(t) = (Z - Rz)(t) + (Rz - z)(t) 
  \end{equation}
  By definition of the discrete and continuous problems
  \begin{equation}    
    \begin{split}
      \alpha  \ltwop{\partial_t \theta_U}{\Phi}_{\mathcal{H}}
      +
      \alpha a_h(\theta_U, \Phi)
      +
      \alpha s_h(\theta_U, \Phi)
      &=
      \ltwop{Z}{\Phi}_{\mathcal{H}}
      - \ltwop{\partial_t (Ru)}{\Phi}_{\mathcal{H}}
      - \alpha a_h(R u, \Phi) - \alpha s_h(R u, \Phi)
      \\
      &=
      \ltwop{Z- Rz}{\Phi}_{\mathcal{H}}
      - \alpha \ltwop{\partial_t (Ru-u)}{\Phi}_{\mathcal{H}}       
    \end{split}
  \end{equation}
  and 
  \begin{equation}    
    \begin{split}
      -\ltwop{\partial_t \theta_Z}{\Psi}_{\mathcal{H}}
      +
      a_h^*(\theta_Z, \Psi)
      +
      s_h^*(\theta_Z, \Psi)
      &=
      \ltwop{\mathcal D - U}{\Psi}_{\mathcal{H}}
      + \ltwop{\partial_t (Rz)}{\Psi}_{\mathcal{H}}
      - a_h^*(R z, \Psi) - s_h^*(R z, \Psi)
      \\
      &=
      \ltwop{Ru- U}{\Psi}_{\mathcal{H}}
      +
      \ltwop{\partial_t (Rz - z)}{\Psi}_{\mathcal{H}}       
    \end{split}
  \end{equation}
  Using the decompositions provided, we take $\Phi = \theta_U$
  \begin{equation}
    \begin{split}
      \alpha \ltwop{\partial_t \theta_U}{\theta_U}_{\mathcal{H}}
    + \alpha a_h(\theta_U, \theta_U)
    + \alpha s_h(\theta_U, \theta_U)
    &= \ltwop{\theta_Z}{\theta_U}_{\mathcal{H}}
    - \alpha \ltwop{\partial_t \rho_U}{\theta_U}_{\mathcal{H}}.
\end{split}
\end{equation}
Now, with $\Psi = \theta_Z$
\begin{equation}
\begin{split}
  -\ltwop{\partial_t \theta_Z}{\theta_Z}_{\mathcal{H}}
  + a_h^*(\theta_Z, \theta_Z)
  + s_h^*(\theta_Z, \theta_Z)
  &=
  -\ltwop{\theta_U}{\theta_Z}_{\mathcal{H}}
  + \ltwop{\partial_t \rho_Z}{\theta_Z}_{\mathcal{H}}.
\end{split}
\end{equation}
Adding the two equations, we obtain
\begin{multline}
  \alpha \ltwop{\partial_t \theta_U}{\theta_U}_{\mathcal{H}}
  -
  \ltwop{\partial_t \theta_Z}{\theta_Z}_{\mathcal{H}}
  +
  \alpha a_h(\theta_U, \theta_U)
  +
  \alpha s_h(\theta_U, \theta_U)
  +
  a_h^*(\theta_Z, \theta_Z)
  +
  s_h^*(\theta_Z, \theta_Z)
  =
  -
  \alpha \ltwop{\partial_t \rho_U}{\theta_U}_{\mathcal{H}}
  +
  \ltwop{\partial_t \rho_Z}{\theta_Z}_{\mathcal{H}}.
\end{multline}
Using the Cauchy-Schwarz and Young inequalities gives
  \begin{equation}
    \frac{1}{2} \ddt 
    \qp{\alpha
      \Norm{\theta_U}_{\mathcal{H}}^2
      -
      \Norm{\theta_Z}_{\mathcal{H}}^2
    }
    +
    \alpha \delta_h \tnorm{\theta_U}^2
    +
    \delta_h \tnorm{\theta_Z}_*^2
    \leq
    \frac{\alpha}{2} \Norm{\partial_t \rho_U}_{\mathcal{H}}^2
    +
    \frac{1}{2} \Norm{\partial_t \rho_Z}_{\mathcal{H}}^2,
  \end{equation}
and integrating from $t = 0$ to $t = T$ then leads to
\begin{equation}
  \begin{split}
    \frac{1}{2}
    \qp{\alpha
      \Norm{\theta_U(T)}_{\mathcal{H}}^2
      +
      \Norm{\theta_Z(0)}_{\mathcal{H}}^2
      }
    +
    \delta_h \qp{\int_0^T \alpha  \tnorm{\theta_U}^2
    +
    \tnorm{\theta_Z}^2 \d t}
    &=    
    \frac{1}{2}
    \qp{
      \int_0^T \alpha\Norm{\partial_t \rho_U}_{\mathcal{H}}^2
      +
      \Norm{\partial_t \rho_Z}_{\mathcal{H}}^2 \d t
    }.
\end{split}
\end{equation}
Corollary \ref{cor:best-approx} then implies
\begin{align*}
  \Norm{\partial_t \rho_U}_{\mathcal{H}}^2
  &\leq
  C h^{2\min(r-1, s-2)} \Norm{\partial_t u}_{\sobh{s}(\Omega)}^2,
  \\
  \Norm{\partial_t \rho_Z}_{\mathcal{H}}^2
  &\leq
  C h^{2\min(r-1, s-2)} \Norm{\partial_t z}_{\sobh{s}(\Omega)}^2.
\end{align*}
Finally, note that:
\begin{equation}
  \Norm{U(T) - u(T)}_{\mathcal{H}}
  +
  \Norm{Z(0) - z(0)}_{\mathcal{H}}  
  \leq
  \Norm{\theta_U(T)}_{\mathcal{H}}
  +
  \Norm{\rho_U(T)}_{\mathcal{H}}
  +
  \Norm{\theta_Z(0)}_{\mathcal{H}}
  +
  \Norm{\rho_Z(0)}_{\mathcal{H}},
\end{equation}
concluding the proof.

\end{proof}

\section{Numerical results}
\label{sec:numerics}

In this section we conduct numerical studies (implemented in the
finite element software FEniCS
\cite{AlnaesLoggOlgaardRognesWells:2014}) to validate the convergence
of the presented methods and to
examine their behaviour in physically-motivated test cases. Unless otherwise stated, the spatial domain is $\left(-1, 1\right)\times\left(-1, 1\right)$, the penalty parameter is taken as $C_\sigma = 10$, and we choose
\begin{equation}
    \vec M = \epsilon\begin{bmatrix} m^3 & m^2 \\ m^2 & m \end{bmatrix},
\end{equation}
with $m = 0.35$.

\begin{example}[Primal problem benchmark]
\label{ex:forward-problem-benchmark}
The first of our numerical experiments demonstrates the expected
convergence behaviour for the discretisation (\ref{eq:spatial-disc})
of the stationary forward problem. We consider the solution
$u\!\left(x\right) := sin^4\!\left(\pi x\right)sin^4\!\left(\pi
v\right)$, with appropriately chosen forcing such that
(\ref{eq:Kolmogorov}) holds over the spatial domain.

Figure \ref{fig:forward-problem-benchmark} shows solution errors for
polynomial degrees $r = 2, 3, 4$ measured in $\leb2(\W)$ and $\sobh
1(\W)$ for a sequence of uniformly-refined triangular meshes with mesh
parameter $h^{-1} = \left\{4, 8, 16, 32, 64\right\}$. The convergence
order agrees with the findings of Corollary \ref{cor:best-approx}.
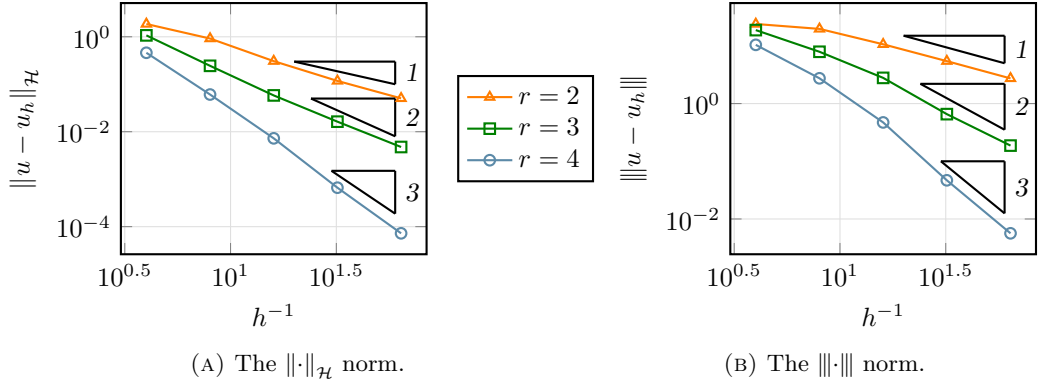
\begin{figure}[h]
  \centering
  \subcaptionbox{The $\Norm{\cdot}_{\mathcal{H}}$ norm.}
    {\begin{tikzpicture}
  \begin{axis}[
  scale = 0.75,
      cycle list/Dark2,
      thick,
      xmode=log,
      ymode=log,
      xlabel=$h^{-1}$,
      ylabel=$\Norm{u - u_h}_{\mathcal{H}}$,
      grid=both,
      minor grid style={gray!25},
      major grid style={gray!25},
      legend style={at={(1.1, 0.5)},anchor=west},
    ]
    \addplot+[color=ptwo, mark = triangle]
    table[x expr={1.0/\thisrow{h}},y=u mathcal_H, col sep=comma]{data/primal_errors_p2.csv};
    \addlegendentry{$r=2$};
    
    \addplot+[color=pthree, mark = square]
    table[x expr={1.0/\thisrow{h}},y=u mathcal_H, col sep=comma]{data/primal_errors_p3.csv};
    \addlegendentry{$r=3$};
    
    \addplot+[color=pfour, mark = o]
    table[x expr={1.0/\thisrow{h}},y=u mathcal_H, col sep=comma]{data/primal_errors_p4.csv};
    \addlegendentry{$r=4$};
    
    \ConvergenceTriangle{20}{3e-1}{3}{1};
    \ConvergenceTriangle{24}{5e-2}{2.5}{2};
    \ConvergenceTriangle{30}{1.5e-3}{2}{3};
  \end{axis}
\end{tikzpicture}}
\subcaptionbox{The $\tnorm{\cdot}$ norm.}
    {\begin{tikzpicture}
  \begin{axis}[
  scale = 0.75,
      cycle list/Dark2,
      thick,
      xmode=log,
      ymode=log,
      xlabel=$h^{-1}$,
      ylabel=$\tnorm{u - u_h}$,
      grid=both,
      minor grid style={gray!25},
      major grid style={gray!25},
      legend style={at={(1.0, 0.0)}, anchor=south east},
    ]
    \addplot+[color=ptwo, mark = triangle]
    table[x expr={1.0/\thisrow{h}},y=u triple_norm, col sep=comma]{data/primal_errors_p2.csv};
    
    \addplot+[color=pthree, mark = square]
    table[x expr={1.0/\thisrow{h}},y=u triple_norm, col sep=comma]{data/primal_errors_p3.csv};
    
    \addplot+[color=pfour, mark = o]
    table[x expr={1.0/\thisrow{h}},y=u triple_norm, col sep=comma]{data/primal_errors_p4.csv};
    
    \ConvergenceTriangle{20}{15}{3}{1};
    \ConvergenceTriangle{24}{2.2}{2.5}{2};
    \ConvergenceTriangle{30}{0.1}{2}{3};
  \end{axis}
\end{tikzpicture}}
    \caption{Plots of the primal problem solution errors for the
      benchmark tests in Example \ref{ex:forward-problem-benchmark},
      using polynomial degrees $r = 2, 3, 4$. Straight dotted lines
      are plotted, along with their gradients, to aid in demonstrating
      the order of convergence.}
    \label{fig:forward-problem-benchmark}
\end{figure}
\end{example}

\begin{example}[Optimal control problem benchmark]
\label{ex:optimal-control-benchmark}
In the second experiment we verify the convergence of the stationary
discrete optimal control problem (\ref{eq:discreteKKT}), with
regularisation parameter $\alpha = 1$. For this example we again
choose $u\!\left(x\right) := sin^4\!\left(\pi x\right)sin^4\!\left(\pi
v\right)$, and the corresponding exact dual solution $z$ and target
dose $\mathcal{D}$ are calculated accordingly. The spatial mesh is
refined as in the previous example, and the resulting errors, measured
in the $\mathcal{H}$ norm, are plotted in Figure
\ref{fig:optimal-control-benchmark}. Convergence rates agree with the
analysis in Corollary \ref{cor:best-approx-optimal-control}.
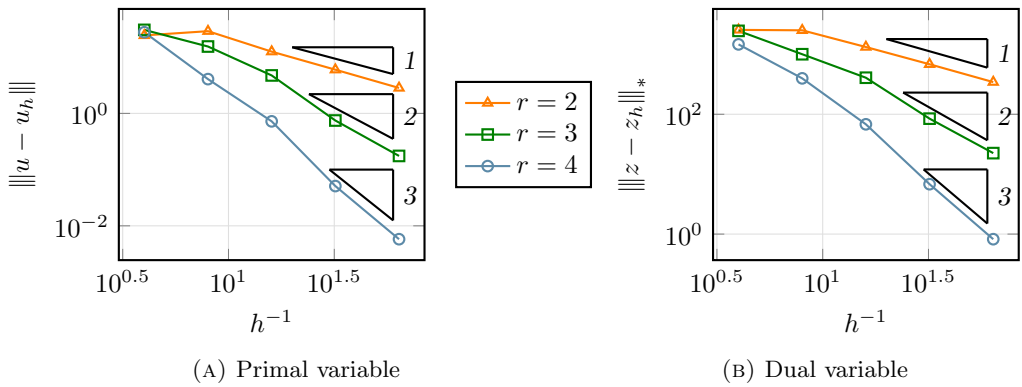
\begin{figure}[h]
\centering
\subcaptionbox{Primal variable}
    {\begin{tikzpicture}
  \begin{axis}[
  scale = 0.75,
      cycle list/Dark2,
      thick,
      xmode=log,
      ymode=log,
      xlabel=$h^{-1}$,
      ylabel=$\tnorm{u - u_h}$,
      grid=both,
      minor grid style={gray!25},
      major grid style={gray!25},
      legend style={at={(1.1, 0.5)},anchor=west},
    ]
    \addplot+[color=ptwo, mark = triangle]
    table[x expr={1.0/\thisrow{h}},y=u triple_norm, col sep=comma]{data/optimal_control_errors_p2.csv};
    \addlegendentry{$r=2$};
    
    \addplot+[color=pthree, mark = square]
    table[x expr={1.0/\thisrow{h}},y=u triple_norm, col sep=comma]{data/optimal_control_errors_p3.csv};
    \addlegendentry{$r=3$};
    
    \addplot+[color=pfour, mark = o]
    table[x expr={1.0/\thisrow{h}},y=u triple_norm, col sep=comma]{data/optimal_control_errors_p4.csv};
    \addlegendentry{$r=4$};
    
    \ConvergenceTriangle{20}{15}{3}{1};
    \ConvergenceTriangle{24}{2.2}{2.5}{2};
    \ConvergenceTriangle{30}{0.1}{2}{3};
  \end{axis}
\end{tikzpicture}}
\subcaptionbox{Dual variable}
    {\begin{tikzpicture}
  \begin{axis}[
  scale = 0.75,
      cycle list/Dark2,
      thick,
      xmode=log,
      ymode=log,
      xlabel=$h^{-1}$,
      ylabel=$\tnorm{z - z_h}_*$,
      grid=both,
      minor grid style={gray!25},
      major grid style={gray!25},
      legend style=none,
    ]
    \addplot+[color=ptwo, mark = triangle]
    table[x expr={1.0/\thisrow{h}},y=z triple_norm, col sep=comma]{data/optimal_control_errors_p2.csv};
    
    \addplot+[color=pthree, mark = square]
    table[x expr={1.0/\thisrow{h}},y=z triple_norm, col sep=comma]{data/optimal_control_errors_p3.csv};
    
    \addplot+[color=pfour, mark = o]
    table[x expr={1.0/\thisrow{h}},y=z triple_norm, col sep=comma]{data/optimal_control_errors_p4.csv};
    
    \ConvergenceTriangle{20}{1.8e3}{3}{1};
    \ConvergenceTriangle{24}{2.3e2}{2.5}{2};
    \ConvergenceTriangle{30}{1.2e1}{2}{3};
  \end{axis}
\end{tikzpicture}}
    \caption{Plots of the primal and dual problem solution errors,
      measured in the $\mathcal{H}$ norm, for the benchmark tests in
      Example \ref{ex:optimal-control-benchmark}, using polynomial
      degrees $r = 2, 3, 4$. Straight dotted lines are plotted, along
      with their gradients, to aid in demonstrating the order of
      convergence.}
    \label{fig:optimal-control-benchmark}
\end{figure}
\end{example}

\begin{example}[Regularisation parameter dependence]
\label{ex:alpha-dependence}
In the third example we consider the behaviour of the stationary
discrete optimal control problem (\ref{eq:discreteKKT}) in the context
of two physically-inspired test cases, with chosen target functions
shown in Figure \ref{fig:target-doses}. The first target function,
which is continuous but has discontinuous first derivative, is given
by
\begin{equation}
  \mathcal{D}_1 = \begin{cases} \sqrt{1 - \frac{x^2}{0.3^2} - \frac{v^2}{0.5^2}}, & 1 - \frac{x^2}{0.3^2} - \frac{v^2}{0.5^2} \geq 0, \\ 0, & \text{otherwise.} \end{cases}
\end{equation}
This is constructed such that it fits into the framework of the
analysis presented. The second target function, which is only
piecewise constant, does not fit into the analytical framework. It is
given by
\begin{equation}
  \mathcal{D}_2 = \begin{cases} 1, & x^2 + v^2 \leq \frac{1}{16}, \\ 0, & \text{otherwise.} \end{cases}
\end{equation}

The target functions could represent, for example, target radiation
doses in a particular energy range to be delivered to a tumour
surrounded by healthy tissue (where we want to minimise
irradiation). However, the formulation in question is overly
simplified for capturing the behaviours of such a complex problem and
should only be considered a toy model rather than an accurate
depiction of the underlying physics. 
\begin{figure}[h]
    \centering
    \subcaptionbox{\label{fig:target-dose-bump}
      A $\cont{0}(\W) \setminus \cont{1}(\W)$ target.
    }
   {\includegraphics[height=0.25\textwidth]{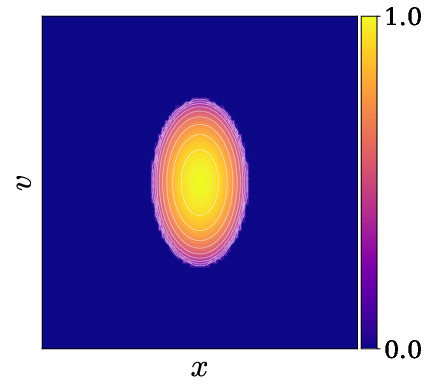}}
   \subcaptionbox{\label{fig:target-dose-circle}
     A discontinuous target.
   }
   {\includegraphics[height=0.25\textwidth]{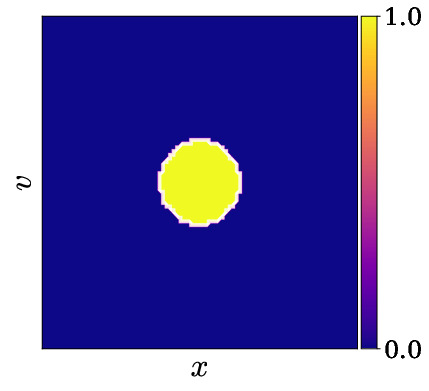}}
    \caption{Target function profiles for the physically-motivated
      experiments in Example \ref{ex:alpha-dependence}.}
    \label{fig:target-doses}
\end{figure}

Using degree-two polynomials on a uniform triangulation with mesh
parameter $h^{-1}=90$ (approximately $65,500$ degrees of freedom), the
discrete optimal control problem (\ref{eq:control-spatial-disc}) is
solved for various values of $\alpha$ in the range $\left[10^{-4},
  10^{-1}\right]$. Two regimes are considered, with relatively small
diffusion ($\epsilon = 10^{-4}$) and relatively large diffusion
($\epsilon = 10^{-1}$), and snapshots of the resulting solutions are
shown in Figures \ref{fig:alpha-dependence-bump-large-diffusion} --
\ref{fig:alpha-dependence-circle-small-diffusion}. The results
demonstrate the smoothing effect of the regularisation parameter
$\alpha$, with the primal variable approaching the target dose as
$\alpha$ is decreased. The effect is more pronounced in the large
diffusion regime (see Figure
\ref{fig:alpha-dependence-bump-large-diffusion} and Figure
\ref{fig:alpha-dependence-circle-large-diffusion}), and particularly
in the case of the piecewise constant target function. In the small
diffusion regime (see Figure
\ref{fig:alpha-dependence-bump-small-diffusion} and Figure
\ref{fig:alpha-dependence-circle-small-diffusion}) the primal solution
closely resembles the target function for moderate $\alpha$.

\begin{figure}[h]
    \centering
    \subcaptionbox{$\alpha = 10^{-1}$}
    {\begin{subfigure}{0.2\textwidth}
    \includegraphics[width=\textwidth]{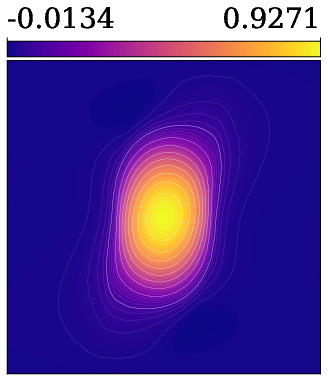}
    \includegraphics[width=\textwidth]{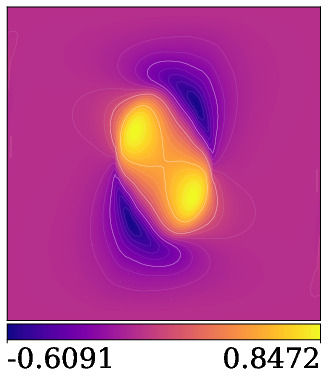}
    \end{subfigure}}
    \subcaptionbox{$\alpha = 10^{-2}$}
    {\begin{subfigure}{0.2\textwidth}
    \includegraphics[width=\textwidth]{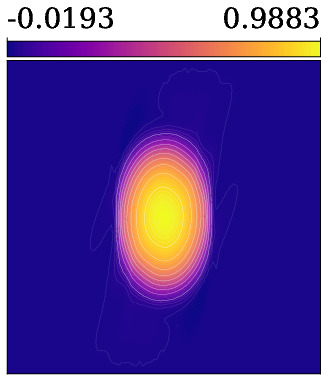}
    \includegraphics[width=\textwidth]{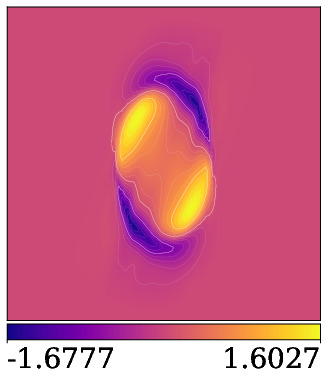}
    \end{subfigure}}
    \subcaptionbox{$\alpha = 10^{-3}$}
    {\begin{subfigure}{0.2\textwidth}
    \includegraphics[width=\textwidth]{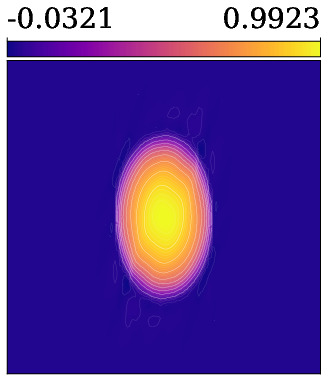}
    \includegraphics[width=\textwidth]{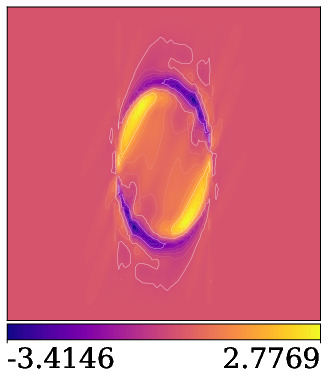}
    \end{subfigure}}
    \subcaptionbox{$\alpha = 10^{-4}$}
    {\begin{subfigure}{0.2\textwidth}
    \includegraphics[width=\textwidth]{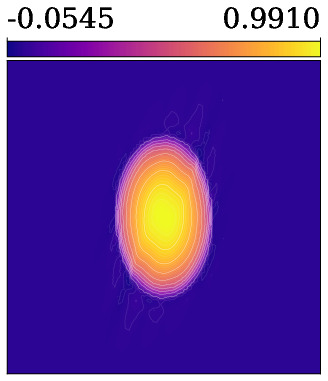}
    \includegraphics[width=\textwidth]{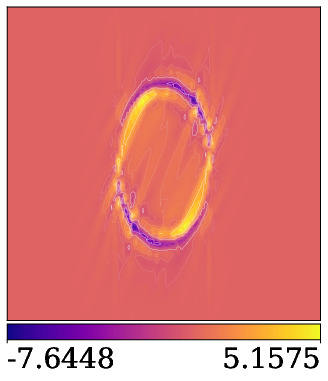}
    \end{subfigure}}
    \caption{Profiles of the primal (top) and control (bottom) variables for various values of $\alpha$ in the large diffusion regime ($\epsilon = 10^{-1}$), using the first target function $\mathcal{D}_1$ in Example \ref{ex:alpha-dependence}.} 
    \label{fig:alpha-dependence-bump-large-diffusion}
\end{figure}

\begin{figure}[h]
    \centering
    \subcaptionbox{$\alpha = 10^{-1}$}
    {\begin{subfigure}{0.2\textwidth}
    \includegraphics[width=\textwidth]{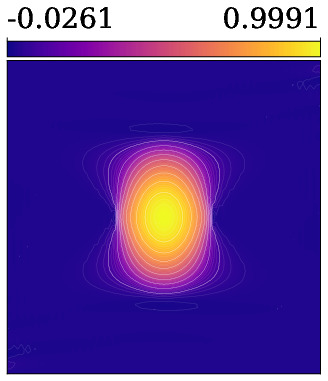}
    \includegraphics[width=\textwidth]{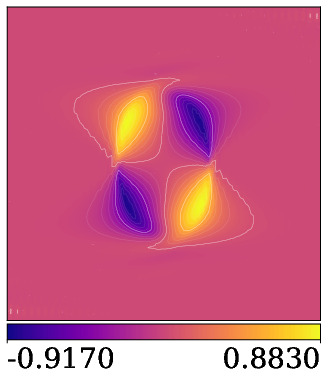}
    \end{subfigure}}
    \subcaptionbox{$\alpha = 10^{-2}$}
    {\begin{subfigure}{0.2\textwidth}
    \includegraphics[width=\textwidth]{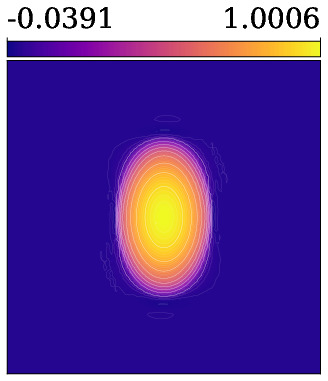}
    \includegraphics[width=\textwidth]{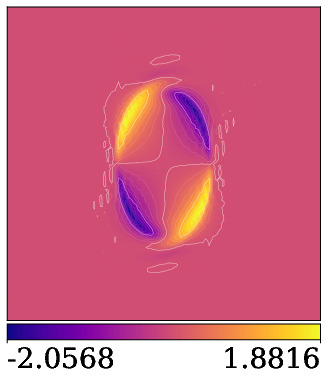}
    \end{subfigure}}
    \subcaptionbox{$\alpha = 10^{-3}$}
    {\begin{subfigure}{0.2\textwidth}
    \includegraphics[width=\textwidth]{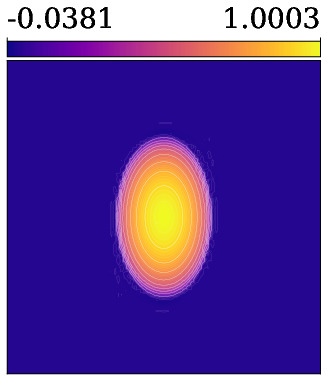}
    \includegraphics[width=\textwidth]{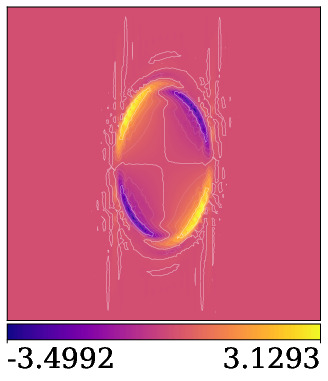}
    \end{subfigure}}
    \subcaptionbox{$\alpha = 10^{-4}$}
    {\begin{subfigure}{0.2\textwidth}
    \includegraphics[width=\textwidth]{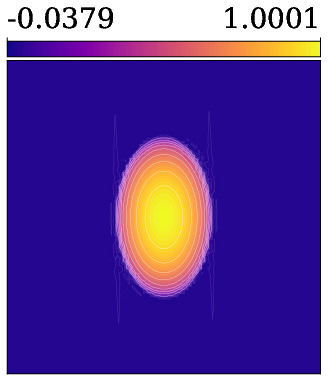}
    \includegraphics[width=\textwidth]{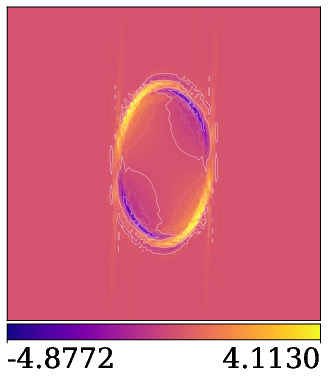}
    \end{subfigure}}
    \caption{Profiles of the primal (top) and control (bottom) variables for various values of $\alpha$ in the small diffusion regime ($\epsilon = 10^{-4}$), using the first target function $\mathcal{D}_1$ in Example \ref{ex:alpha-dependence}.} 
    \label{fig:alpha-dependence-bump-small-diffusion}
\end{figure}

\begin{figure}[h]
    \centering
    \subcaptionbox{$\alpha = 10^{-1}$}
    {\begin{subfigure}{0.2\textwidth}
    \includegraphics[width=\textwidth]{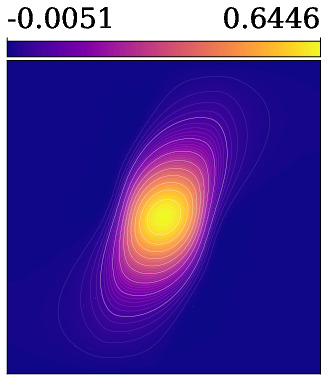}
    \includegraphics[width=\textwidth]{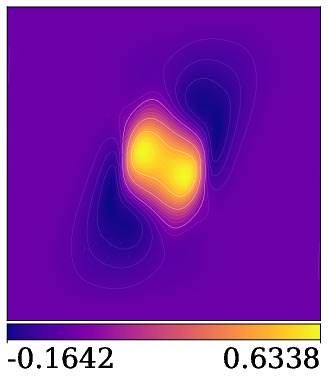}
    \end{subfigure}}
    \subcaptionbox{$\alpha = 10^{-2}$}
    {\begin{subfigure}{0.2\textwidth}
    \includegraphics[width=\textwidth]{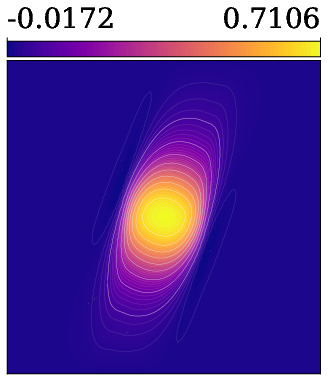}
    \includegraphics[width=\textwidth]{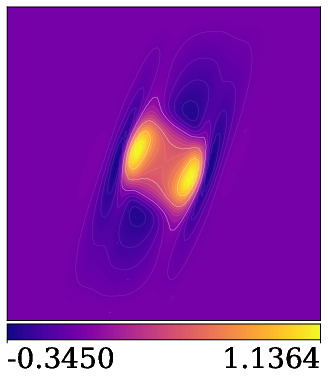}
    \end{subfigure}}
    \subcaptionbox{$\alpha = 10^{-3}$\label{fig:alpha-dep-circle-big-three}}
    {\begin{subfigure}{0.2\textwidth}
    \includegraphics[width=\textwidth]{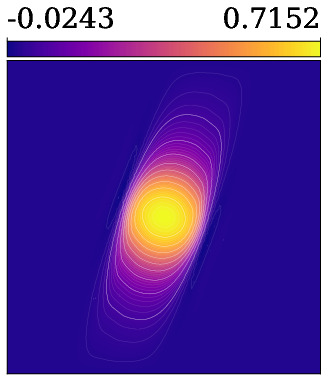}
    \includegraphics[width=\textwidth]{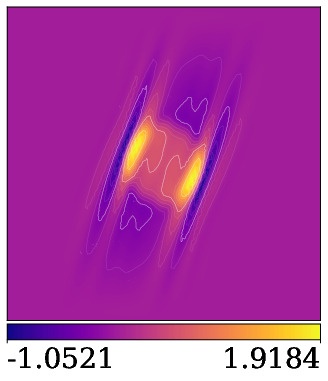}
    \end{subfigure}}
    \subcaptionbox{$\alpha = 10^{-4}$}
    {\begin{subfigure}{0.2\textwidth}
    \includegraphics[width=\textwidth]{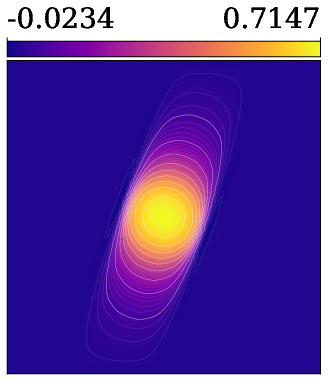}
    \includegraphics[width=\textwidth]{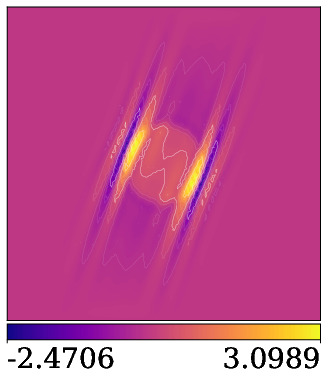}
    \end{subfigure}}
    \caption{Profiles of the primal (top) and control (bottom) variables for various values of $\alpha$ in the large diffusion regime ($\epsilon = 10^{-1}$), using the second target function $\mathcal{D}_2$ in Example \ref{ex:alpha-dependence}.
    } 
    \label{fig:alpha-dependence-circle-large-diffusion}
\end{figure}

\begin{figure}[h]
    \centering
    \subcaptionbox{$\alpha = 10^{-1}$}
    {\begin{subfigure}{0.2\textwidth}
    \includegraphics[width=\textwidth]{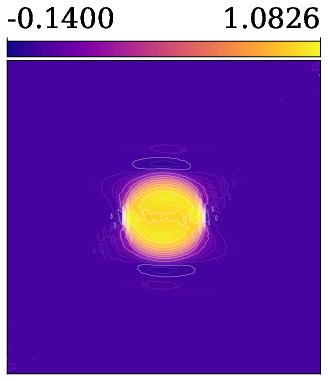}
    \includegraphics[width=\textwidth]{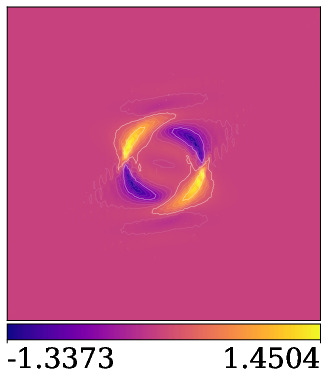}
    \end{subfigure}}
    \subcaptionbox{$\alpha = 10^{-2}$}
    {\begin{subfigure}{0.2\textwidth}
    \includegraphics[width=\textwidth]{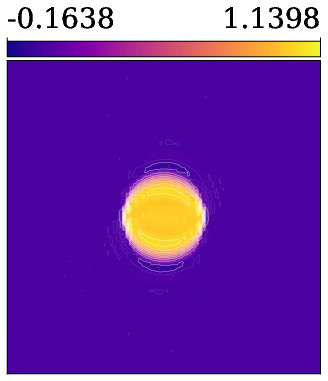}
    \includegraphics[width=\textwidth]{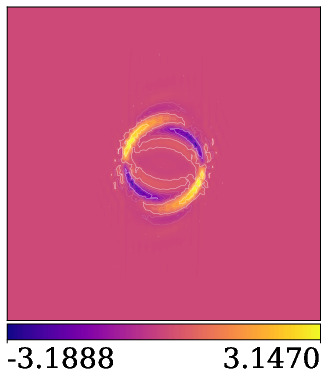}
    \end{subfigure}}
    \subcaptionbox{$\alpha = 10^{-3}$}
    {\begin{subfigure}{0.2\textwidth}
    \includegraphics[width=\textwidth]{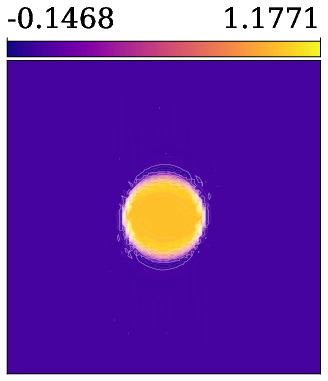}
    \includegraphics[width=\textwidth]{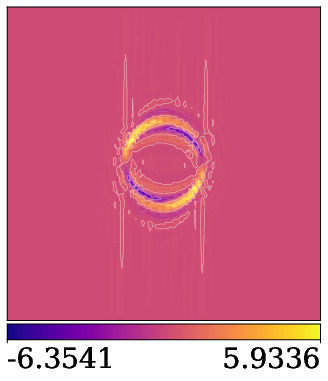}
    \end{subfigure}}
    \subcaptionbox{$\alpha = 10^{-4}$}
    {\begin{subfigure}{0.2\textwidth}
    \includegraphics[width=\textwidth]{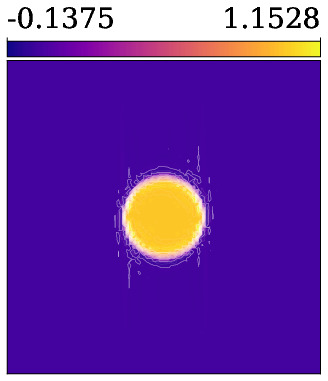}
    \includegraphics[width=\textwidth]{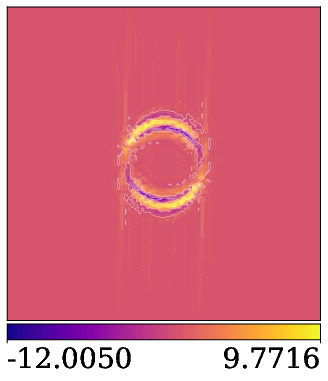}
    \end{subfigure}}
    \caption{Profiles of the primal (top) and control (bottom) variables for various values of $\alpha$ in the small diffusion regime ($\epsilon = 10^{-4}$), using the second target function $\mathcal{D}_2$ in Example \ref{ex:alpha-dependence}.
    } 
    \label{fig:alpha-dependence-circle-small-diffusion}
\end{figure}
\end{example}

\begin{example}[Choice of the matrix $\vec M$]
    \label{ex:choice-of-M}  
    The impact of the choice of the matrix $\vec M$ is examined. Using
    degree-two polynomials on the same mesh as in Example
    \ref{ex:alpha-dependence}, the stationary version of the discrete
    optimal control problem (\ref{eq:control-spatial-disc}) is solved
    for various values of $m$ in the range $\left[0, 10\right]$, with
    target function $\mathcal{D}_2$ in the large diffusion regime and
    the regularisation parameter set as $\alpha = 10^{-3}$. Snapshots
    of the solution are shown in Figure \ref{fig:M-dependence}, from
    which the non-isotropic diffusion introduced by the matrix $\vec
    M$ (the hypocoercive diffusion) can be clearly visualised. As $m$
    increases, the primal solution becomes increasingly diffuse, and
    the diffusion tends towards the $x$ direction. As $m$ approaches
    zero (the non-hypocoercive formulation) there is no additional
    diffusion introduced and the solution is closer to the target
    function. This potentially motivates the choice of mesh dependant
    $\vec M = \vec M(h)$.
    \begin{figure}[h]
    \centering
    \subcaptionbox{$m= 10^{-1}$}
    {\begin{subfigure}{0.2\textwidth}
    \includegraphics[width=\textwidth]{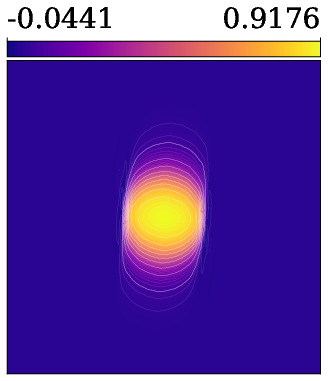}
    \includegraphics[width=\textwidth]{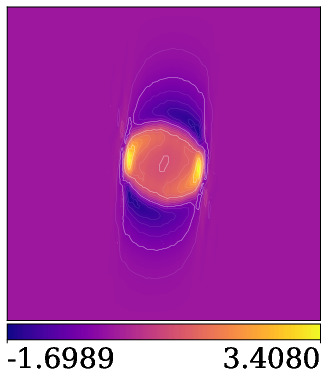}
    \end{subfigure}}
    \subcaptionbox{$m = 10^{-0.5}$}
    {\begin{subfigure}{0.2\textwidth}
    \includegraphics[width=\textwidth]{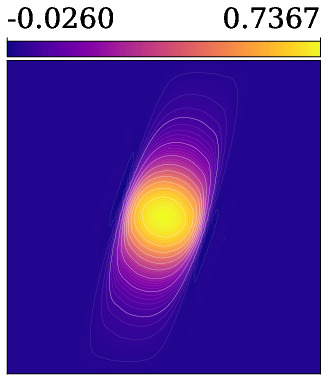}
    \includegraphics[width=\textwidth]{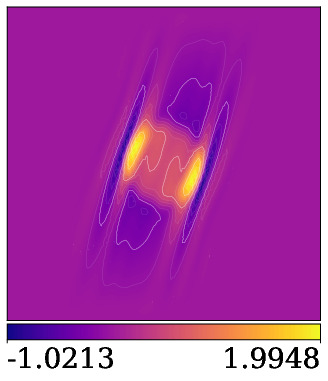}
    \end{subfigure}}
    \subcaptionbox{$m = 1$}
    {\begin{subfigure}{0.2\textwidth}
    \includegraphics[width=\textwidth]{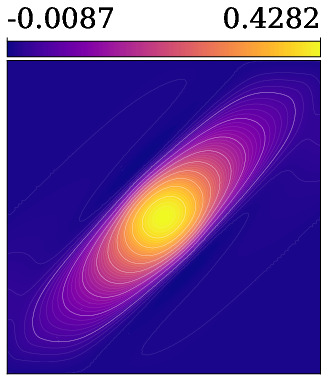}
    \includegraphics[width=\textwidth]{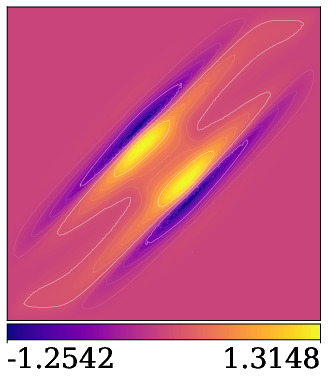}
    \end{subfigure}}
    \subcaptionbox{$m = 10^{0.5}$}
    {\begin{subfigure}{0.2\textwidth}
    \includegraphics[width=\textwidth]{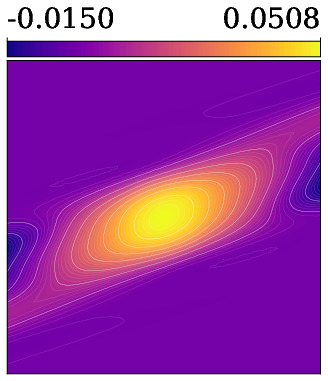}
    \includegraphics[width=\textwidth]{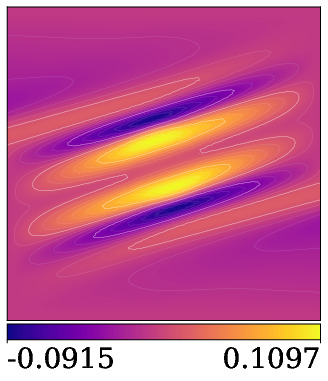}
    \end{subfigure}}
    \caption{Profiles of the primal (top) and control (bottom) variables for various values of $m$ in the large diffusion regime ($\epsilon = 10^{-1}$), using the target function $\mathcal{D}_2$ in Example \ref{ex:choice-of-M}.
    } 
    \label{fig:M-dependence}
\end{figure}
\end{example}

\begin{example}[Bound-satisfying optimal control]
  \label{ex:bound-satisfying}
  Results from the previous examples demonstrate that the control
  variable obtained by solving the discrete optimal control problem
  (\ref{eq:discreteKKT}) can take negative values. In applications
  such as radiotherapy this would correspond to removing radiation at
  that point, which is non-physical. Since the degree of the finite
  element space is $r > 1$, it is non-trivial to design a method
  consistent pointwise with such a constraint. In view of this we pose
  a modified method, based on the nodal (at the degrees of freedom)
  projection approach introduced in
  \cite{Barrenechea:2023,AmiriBarrenecheaPryer:2024}, to enforce that
  the control variable additionally satisfies imposed box constraints
  (more general bounds could also be considered).

\newcommand{\boundfes}{\fes_{\mathcal{B}}}
\newcommand{\iterstep}{\omega}
\newcommand{\itertol}{\operatorname{tol}}
\newcommand{\node}{\vec x_i}
\newcommand{\nodalproj}[1]{\mathcal{P}_{\mathcal{B}}\!\left( #1 \right)}

We denote the nodes of the finite element space $\fes$ by $\node$, for $i = 1,\ldots,N$, with $\phi_1,\ldots,\phi_N$ being the associated Lagrange basis functions. The nodally bound-satisfying admissible set $\boundfes\subseteq\fes$, given by
\begin{equation} \label{eq:boundfesdef}
    \boundfes := \left\{v\in\fes : v\!\left(\node\right)\in\left[0, \kappa\right]\text{, for all }i = 1,\ldots,N\right\},
\end{equation}
contains functions that lie between 0 and $\kappa$ at their degrees of freedom. Any finite element function can be projected into this set via the nodal projection operator $\mathcal{P}_{\mathcal{B}}: \fes\longrightarrow\boundfes$, which is defined, for all $v\in\fes$, by
\begin{equation}
    \nodalproj{v} := \sum^N_{i=1}\left(\max\left(0, \min\left(v\!\left(\node\right), \kappa\right)\right)\phi_i\right).
\end{equation}
Using this notation we pose the following Richardson-type iterative scheme, which enforces $F = \frac{1}{\alpha}\nodalproj{Z}$ in the ``steady-state" limit and decouples the system, allowing the primal and dual equations to be solved sequentially. Let $\iterstep\in\left(0, 1\right]$ and $n\in\mathbb{N}$. Given $\qp{U^{n-1}, Z^{n-1}}\in\fes^-_0\times\fes^+_0$, seek $\qp{U^n, Z^n}\in\fes^-_0\times\fes^+_0$, satisfying
\begin{align}
    a_h(U^n, \Phi) + s_h(U^n, \Phi) &= \frac{\iterstep}{\alpha}\ltwop{\nodalproj{Z^n}}{\Phi}_\mathcal{H} + \left(1 - \iterstep\right)\left(a_h(U^{n-1}, \Phi) + s_h(U^{n-1}, \Phi)\right)\Foreach\Phi\in\fes^-_0, \label{eq:iteration-primal} \\
    a_h^*(Z^n, \Psi) + s_h^*(Z^n, \Psi) &= \ltwop{\mathcal D - U^{n-1}}{\Psi}_\mathcal{H}\Foreach\Psi\in\fes^+_0. \label{eq:iteration-dual}
\end{align}

The experimental setup is taken to be the same as in Example \ref{ex:choice-of-M}, except we fix $m = 0.35$. Both single-sided and double-sided constraints are considered, replacing $v\!\left(\vec x_i\right)\in\left[0, \kappa\right]$ with $v\!\left(\vec x_i\right)\geq 0$ in (\ref{eq:boundfesdef}) in the former case and taking $\kappa = 1$ in the latter. The iterative scheme \eqref{eq:iteration-primal}--\eqref{eq:iteration-dual} is solved, with $\iterstep = 10^{-3}$, until $\Norm{Z^n - Z^{n-1}}^2 < \itertol := 10^{-10}$, to obtain the results depicted in Figure \ref{fig:bound-satisfying}. The resulting control variables respect the bounds (nodally) at the cost of the primal variables being further from the target function. Comparing to Figure \ref{fig:alpha-dep-circle-big-three}, the primal variables with the imposed control restrictions are more diffuse, featuring additional smearing in the direction of transport. This effect is furthered only slightly in the double-sided constraint case, however, the impact is expected to heavily depend on the choice of $\kappa$.

\begin{figure}[h]
    \centering
    \subcaptionbox{Single-sided constraint}
    {\begin{subfigure}{0.49\textwidth}
    \includegraphics[width=0.49\textwidth]{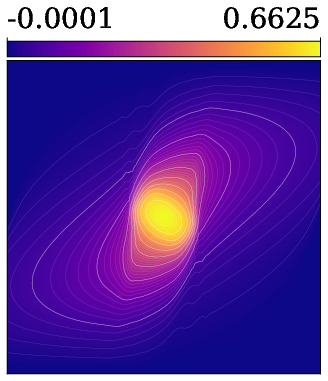}
    \includegraphics[width=0.49\textwidth]{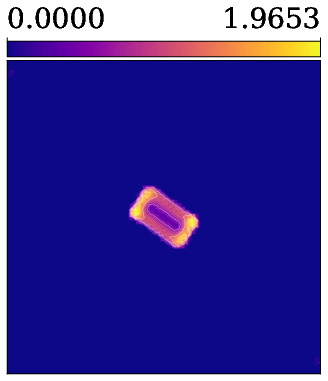}
    \end{subfigure}}
    \subcaptionbox{Double-sided constraint}
    {\begin{subfigure}{0.49\textwidth}
    \includegraphics[width=0.49\textwidth]{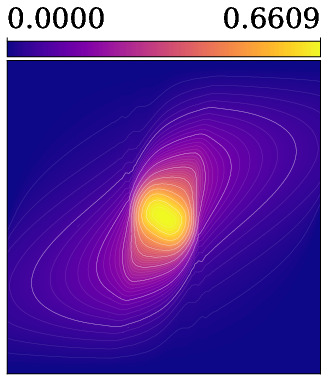}
    \includegraphics[width=0.49\textwidth]{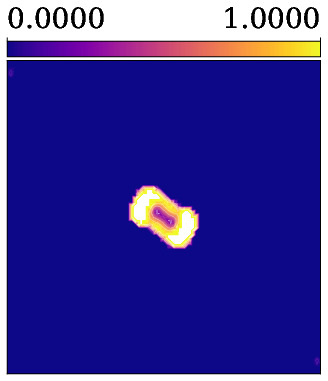}
    \end{subfigure}}
    \caption{Profiles of the primal variable (left in both subfigures), and control variable (right in both subfigures) for the single-sided ($F \geq 0$) and double-sided ($F\in\left[0, 1\right]$) control constraints in Example \ref{ex:bound-satisfying}.
    } 
    \label{fig:bound-satisfying}
\end{figure}

\end{example}

\begin{example}[Time-dependent optimal control]
\label{ex:space-time}

\newcommand{\energy}{\mathscr{E}}
\newcommand{\energyinit}{\energy_0}

In the final example, the time-dependent discrete optimal control problem (\ref{eq:control-spatial-disc}) is considered, with initial condition $u_0 = 0$ and target function
\begin{equation}
    \mathcal{D} = \qp{1 - \cos{2\pi t}}\exp\!\left(-25\left(x^2 + v^2\right)\right).
\end{equation}
Taking $\alpha = 10^{-2}$ and $\epsilon = 10^{-1}$, the solution is obtained using degree-two polynomials on a tetrahedral partition of the space-time domain $\left(-1, 1\right)^2\times\left(0, 1\right)$, with mesh size parameter $h^{-1} = 16$. The resulting primal and control variables are visualised in Figure \ref{fig:space-time-snapshots} as $xv$-plane slices at different values of $t$, and the qualitative resemblance of the primal variable to the target function is clear. There are small discrepancies, such as a lower maximum magnitude and a slight increase at the final time, but this is to be expected with the presence of diffusion. The results from this experiment demonstrate the potential of the time-dependent hypocoercive formulation in the optimal control setting. Box constraints could also be applied in a similar manner to the previous experiment to enforce physical bounds.
\begin{figure}[h]
    \centering
    \subcaptionbox{$t = 0.125$}
    {\begin{subfigure}{0.2\textwidth}
    \includegraphics[width=\textwidth]{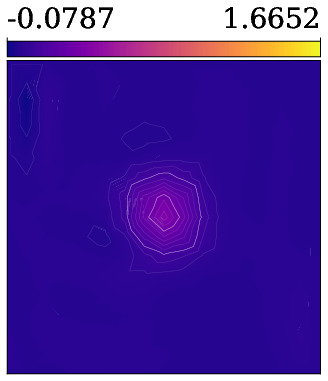}
    \includegraphics[width=\textwidth]{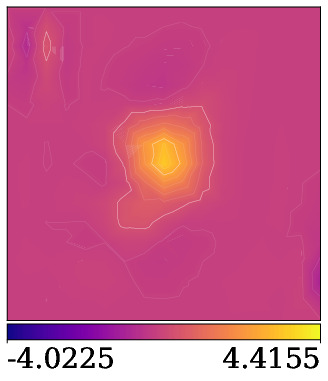}
    \end{subfigure}}
    \subcaptionbox{$t = 0.375$}
    {\begin{subfigure}{0.2\textwidth}
    \includegraphics[width=\textwidth]{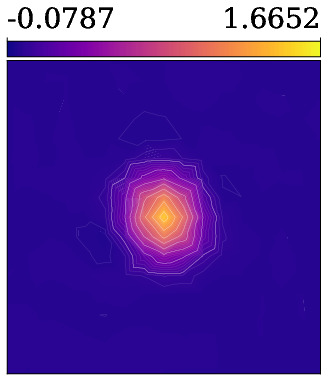}
    \includegraphics[width=\textwidth]{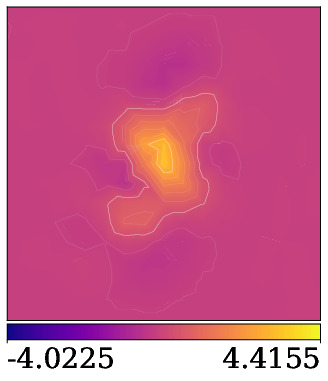}
    \end{subfigure}}
    \subcaptionbox{$t = 0.625$}
    {\begin{subfigure}{0.2\textwidth}
    \includegraphics[width=\textwidth]{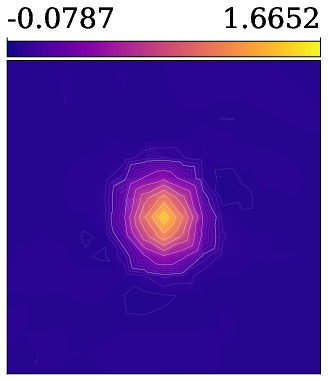}
    \includegraphics[width=\textwidth]{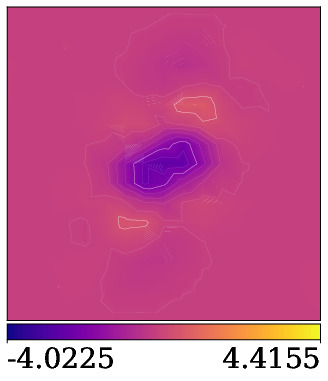}
    \end{subfigure}}
    \subcaptionbox{$t = 1$}
    {\begin{subfigure}{0.2\textwidth}
    \includegraphics[width=\textwidth]{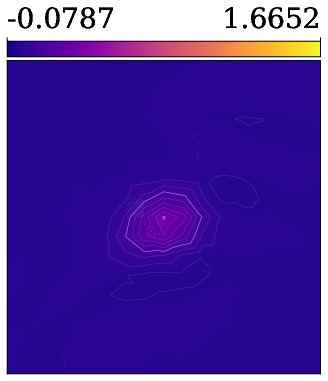}
    \includegraphics[width=\textwidth]{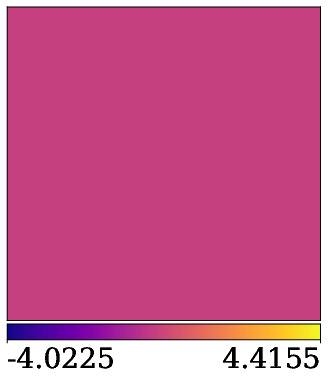}
    \end{subfigure}}
    \caption{Profiles of the primal variable (top) and control variable (bottom) at various values of $t$ in Example \ref{ex:space-time}. The colour scale is kept constant to demonstrate the evolution of the solution.
    } 
    \label{fig:space-time-snapshots}
\end{figure}

\end{example}

\section{Concluding remarks}

In this work, we have introduced a hypocoercive formulation of the
Kolmogorov equation, around which an optimal control problem is
constructed. Due to the hypocoercive nature of the equation, we are
able to show long-time stability of the solutions, which decay in the
hypocoercive norm as the terminal time increases. A
hypocoercivity-preserving finite element method, introduced in
\cite{georgoulis2020hypocoercivitycompatible}, has been extended to
the optimal control setting, and stability in time has been
demonstrated.

Through numerical experiments, we observe that naive approaches to
optimal control of kinetic equations without stabilisation are highly
unstable, underscoring the role of regularisation in designing
numerical methods for such problems. The hypocoercive stabilisation
provides a robust analytic framework that is not always available with
other approaches. We have established stability and best approximation
results within this framework. However, from an approximation theory
perspective, these results are suboptimal by one order. This
limitation arises because the regularisation is not asymptotic; in
particular, the matrix $\vec M$ does not tend to zero as the mesh size
decreases. Investigating compatible regularisations that achieve
asymptotic optimality remains an important direction for future work.

Additionally, the methods developed here can be extended to optimal
control problems with more complex constraints, such as box
constraints, which we have explored in physically motivated numerical
test cases. Moving forward, a natural progression is to apply these
methods to more physically realistic models, such as the
Boltzmann-Fokker-Planck equation for kinetic systems.

\section*{Acknowledgements}
AP and TP received support from the EPSRC programme grant
EP/W026899/1. TP was also supported by the Leverhulme RPG-2021-238 and
EPSRC grant EP/X030067/1. AT is supported by a scholarship from the
EPSRC Centre for Doctoral Training in Advanced Automotive Propulsion
Systems (AAPS), under the project EP/S023364/1.

\printbibliography
\end{document}